\documentclass{article}

\usepackage{amsthm}  

\usepackage{setspace} 

\usepackage{amsmath}
\usepackage{latexsym}
\usepackage{amssymb}
\usepackage{verbatim}
\usepackage{setspace}
\usepackage{qtree}
\usepackage{pbox}
\usepackage[shortlabels]{enumitem}

\newtheorem{df}{Def}[section]  
\newtheorem{teo}[df]{Theorem}
\newtheorem{cor}[df]{Corollary}
\newtheorem{lem}[df]{Lemma}
\newtheorem{lemma}[df]{Lemma}
\newtheorem{conv}[df]{Convention}
\newtheorem{prop}[df]{Proposition}
\newtheorem{example}[df]{Example}

\newtheorem{notation}[df]{Notation}

\newtheorem{conjecture}[df]{Conjecture}

\newcommand{\gap}{[\phantom{a}]}

\newcommand{\CORR}[1]{#1}

\author{Fausto Barbero}

\title{Complexity of syntactical tree fragments of Independence-Friendly logic\footnote{The present work has been developed under the Academy of Finland project 286991, ``Dependence and Independence in Logic: Foundations and Philosophical Significance'', and revised under the Academy of Finland project 316460, ``Semantics of causal and counterfactual dependence''.
}
}

\begin{document}

\setcounter{page}{1}

\maketitle


\begin{abstract}
A dichotomy result of Sevenster (2014) completely classified the quantifier prefixes of regular Independence-Friendly (IF) logic according to the patterns of quantifier dependence they contain. On one hand, prefixes that contain ``Henkin'' or ``signalling'' patterns were shown to characterize fragments of $IF$ logic that capture NP-complete problems; all the remaining prefixes were shown instead to be essentially first-order.

In the present paper we develop the machinery which is needed in order to extend the results of Sevenster to non-prenex, regular IF sentences. This involves shifting attention from quantifier prefixes to a (rather general) class of syntactical tree prefixes. 

We partially classify the fragments of regular $IF$ logic that are thus determined by syntactical trees; in particular, a) we identify three tree prefixes that are neither signalling nor Henkin, and yet express NP-complete problems and other second-order concepts; and b) we give more general criteria for checking the first-orderness of an $IF$ sentence. \\


\end{abstract}








\section{Introduction}

In formulas of first-order logic, an existential quantifier is implicitly dependent on all quantifiers that occur above it. For example, a sentence of the form $\forall x \exists y\psi(x,y)$ asserts that, for each value  that can be picked for $x$, a value for $y$ can be chosen so that $\psi(x,y)$ is satisfied. In other words, the fact that $\exists y$ occurs in the scope of $\forall x$ determines the existence of a (Skolem) function $f$ such that $\psi(x,f(x))$ holds. Independence-Friendly (IF) logic is an extension of first-order logic which frees the notion of dependence between variables from the syntactical notion of scope dependence. This is obtained by enriching the syntax with a slashing device, that is, by allowing quantifiers of the form $(Qv/V)$, where $V$ is a finite set of variables. A quantifier $(\exists v/V)$ expresses the fact that $v$ is to picked as a function of all quantifiers above  $(\exists v/V)$ \emph{except} for those that are listed in $V$. $IF$ logic was introduced in \cite{HinSan1989} with the purpose of decomposing the so-called partially ordered quantifiers of Henkin (\cite{Hen1961}) into individual quantifiers. The simplest among the Henkin quantifiers is a prefix of $4$ classical quantifiers arranged as in the formula

$$\left(\begin{array}{ll} \forall x & \exists y\\  \forall z & \exists w \end{array}\right)\psi(x,y,z,w)$$

\noindent which is meant to assert that $y$ must be picked as a function of $x$ only, and $w$ as a function of $z$ only. The meaning of such a sentence is expressed by a Skolemization of the form $\exists f\exists g\forall x\forall z\psi(x,f(x),z,g(z))$. In the version of $IF$ logic that we consider here (which was introduced in \cite{Hod97}, and is sometimes called \emph{slash logic}) this Henkin sentence can be rendered in a number of different ways, for example by the sentence $\forall x\exists y\forall z(\exists w/\{x,y\})\psi(x,y,z,w)$. However, it has been soon realized that slash logic can express patterns of dependence which differ from those that come from Henkin quantifiers. One example is the short signalling sequence in  
$$\forall x \exists y(\exists z/\{x\})\chi(x,y,z).$$
\noindent If, say, $\chi(x,y,z)$ is a quantifier-free formula, then this sentence is equivalent to a Skolem formula of the form $\exists f\exists g\forall x\forall z \chi(x,f(x),g(f(x)))$. One way to understand what happens in the evaluation of such a formula (over some structure $M$) is to imagine a game played by a Falsifier and a team of two Verifiers, who pick witnesses for $x,y,z$; the purpose of the Falsifier is to build an assignment which does not satisfy $\chi(x,y,z)$, while the Verifiers aim for the opposite goal. The game is in three consecutive turns: first the Falsifier picks a value for $x$; the first Verifier then picks a value for $y$; finally, the second Verifier picks a value for $z$ without looking at the value that was chosen for $x$; if the chosen values satisfy $\chi$, then the Verifiers win. If the Verifiers have a winning strategy for this game, then the sentence is true in the structure under consideration. Notice that the second Verifier can look at what the first Verifier has chosen; the two might then agree to use the value chosen for $y$ in order to signal some information about the value of $x$ (which is visible to the first Verifier, but not to the second). Hence the name of \emph{signalling} for this kind of sequence of quantifiers. It was realized in \cite{CaiKry1999} that this kind of sentence can express second-order notions such as infinity over a poor vocabulary.\footnote{Actually, some suggestion of this kind already occurred in the literature on Henkin quantifiers, see \cite{End1970}.} Later, \cite{Sev2014} showed the same over finite structures; signalling sentences can express the NP-complete problem EXACT COVER BY 3-SET. In   \cite{Sev2014} one also finds a classification result: if we restrict attention to \emph{prenex} $IF$ sentences which are \emph{regular} (that is, variables cannot be requantified), then the quantifier prefixes which can express second-order concepts are exactly those that are Henkin or signalling (in a specific technical sense that will be reviewed later). If instead a prenex, regular $IF$ sentence has a quantifier prefix that is neither Henkin nor signalling, then it can be mechanically transformed into an equivalent first-order sentence.

The purpose of the present paper is to extend the results of \cite{Sev2014} beyond the boundaries of prenex $IF$ logic, by identifying  syntactical structures (tree prefixes) which can capture second order concepts and others which cannot. We are lead in this direction by a number of examples from the literature which point to the presence, in $IF$ logic, of interesting interactions between quantifiers \emph{and connectives}. To give an idea, we review an example from \cite{Jan2002} which illustrates the phenomenon of \emph{signalling by disjunction}. Consider the formula
$$\forall x (\exists y/\{x\})x\neq y.$$
Here a value for $y$ must be chosen to be different from the value of $x$ without knowing the value of $x$; it should be clear that, if the structure under consideration has at least two distinct elements, no strategy allows to do this; thus the sentence is not true on such a structure. However, the sentence 
$$\forall x ( (\exists y/\{x\})x\neq y \lor (\exists y/\{x\})x\neq y )$$
is true in all structures with at least two elements. The reason  can be again explained in game-theoretical terms. Here we have three Verifiers (one for the disjunction, and one for each existential quantifier) and one Falsifier, for the universal quantifier. First the Falsifier picks a value for $x$; then the first Verifier picks one of the two disjuncts; if the left disjunct is chosen, then the Verifier corresponding to the left occurrence of  $(\exists y/\{x\})$ picks a value for $y$; if the witnesses picked for $x$ and $y$ are distinct, then the Verifiers win, and otherwise the Falsifier wins. In case instead the right disjunct is chosen, the game proceeds analogously, but the witness for $y$ is picked by the Verifier associated to the right occurrence of  $(\exists y/\{x\})$.   On structures with at least two elements, the Verifier team has the following winning strategy: fix two distinct elements $a,b$; choose the left disjunct if and only if $a$ was chosen for $x$; choose $b$ for $y$ in the left disjunct; choose $a$ for $y$ in the right disjunct (this ensures that distinct witnesses for $x$ and $y$ are picked in any play).  The interpretation that   \cite{Jan2002} gives of this phenomenon is that some binary information can be stored by the Verifier team by means of the choice of a disjunct, in analogy with signalling by means of the choice of a value for a quantifier. This parallelism lead us to conjecture that some form of signalling by disjunction might be used to express second-order concepts. This is not the case for the syntctical tree of the example above, which we can write as  
$$\forall x ( (\exists y/\{x\})\gap \lor (\exists y/\{x\})\gap);$$
a careful use of equivalence rules of $IF$ logic tranforms any sentence of this form (with the gap symbols $\gap$ replaced by quantifier-free formulas) into a first-order sentence. However, we will see that appropriate extensions of this basic pattern allow expressing NP-complete problems; we identified three such patterns, which appear under the names of GH2($\lor$), C1 and C2 patterns (see e.g. the table at the end of the paper). We must also mention that, recently (\cite{JaaBar2019}), a new second-order pattern has been identified which does not use disjunctions. For completeness, we will briefly describe it (under the name of GH3 pattern). 

We briefly explain why the NP-complete problems play an important role in our attempt of classification. The reason lies in the fact that $IF$ logic (as well as many other \emph{logics of imperfect information}, such as positive Henkin quantification \cite{BlaGur1986}, Dependence logic \cite{Vaa2007}, Inclusion logic \cite{Gal2012}, Independence logic \cite{GraVaa2013}) is expressively equivalent, at the level of sentences, with existential second-order logic (ESO). A  classical descriptive complexity result of Fagin (\cite{Fag1974}) tells us that, on finite structures, ESO captures exactly the complexity class NP of decision problems which can be solved in polynomial time by a nondeterministic Turing machine. The link between logic and complexity is given, in this case, by the ``data complexity'' version of the model checking problem: given a fixed formula $\psi$ expressed in some logical language, and a class $K$ of finite structures, the problem asks whether an input structure $M\in K$ satisfies $\varphi$ ($M\models\varphi$). The choice of reasonable encodings of input instances, and of the class $K$, allows reducing decision problems, that may seem to be completely unrelated to logic, to model checking problems. A decision problem is \emph{described} by a sentence $\varphi$ if there is a reasonable\footnote{The idea behind this notion of ``reasonable'' is a bit vague. Roughly, it means that the encoding  transforms instances of a problem D, in one of its typical presentations, into finite structures without using more resources than what are needed for solving the problem in the encoded form. For most of the discussion in this paper, this amounts to the requirement that the encoding takes polynomial time.} encoding of the instances of the problem into $K$, so that, for all $M\in K$,  $M\models \varphi$ if and only if $M$ encodes a ``yes'' instance of the problem. In this sense, $IF$ logic can be seen as an alternative cartography of the NP complexity class. The NP-complete problems are then important for the purpose of classification of second-order fragments of IF logic, because it is well known that such problems are not expressible in first-order logic. If a fragment contains a description of an NP-complete problem, then the fragment is not translatable into first-order logic, and its corresponding model-checking problem is unfeasible.\footnote{This last assertion is conditional on the well-known open conjecture P$\neq$NP, and on Cobham's thesis -- that the polynomial time complexity class captures the notion of feasibility.} Such a fragment will be itself called NP-hard (or \emph{mighty}, in the terminology of \cite{BlaGur1986}). Descriptions of complete problems for lower complexity classes such as L, NL, P would serve equally well the purpose of identifying second-order fragments; for example, recent work on the classification of prefixes of relational ESO (\cite{GotKolSch2000},\cite{Tan2015}) lead to the discovery of fragments which describe L- or NL-complete problems, but no problems of higher complexity. However, the classification  result of \cite{Sev2014} shows that quantifier prefix fragments of regular $IF$ logic are either NP-hard or in FO (the class of first-order definable problems); in this sense, it was called a \emph{dichotomy result}. The results of our paper, although they do not reach a full classification, seem to confirm that also the tree prefixes fall into this dichotomy.\footnote{This statement concerns those tree prefixes that generalize the notion of quantifier prefix; we will see that a less strict notion of tree prefix allows capturing an NL-complete fragment.} (On the other hand, a recent paper, \cite{BarHelRon2018}, proposed a candidate for an \emph{ir}regular prefix that is not first-order but plausibly stays within the complexity class L.)

The kind of tree prefixes that we are trying to classify are those that are \emph{positive} and \emph{initial}, by which we mean that they contain no occurrences of negation symbols nor of atomic formulas; this seems to be the reasonable analogue of a quantifier prefix, at least in the context of $IF$ logic. To make a concrete example, we will study the fragment of $IF$ sentences of the special form $\forall x(\forall y(\exists u/\{x\})\epsilon_1(x,y,u) \, \, \lor \, \, \forall z(\exists v/\{x\})\epsilon_2(x,z,v))$, where $\epsilon_1(x,y,u)$ and $\epsilon_2(x,z,v)$ are quantifier-free. We want to point out that the study of this fragment is not easily reducible to known results\footnote{The only classification results for \emph{functional} ESO  that we are aware of are those of \cite{Gra1990}, which show that the smallest non-first order prefix of functional ESO, $\exists f\forall x$, already suffices to capture NP-complete problems.} on ESO: the Skolemization procedure transforms sentences of this form into equivalent functional ESO sentences of the (not very simple) form $\exists f\exists g\forall x\forall y\forall z(\epsilon_1(x,y,f(y)/u) \lor \epsilon_2(x,z,g(z)/v))$. Sentences of this form do not fully cover the fragment of ESO corresponding to the quantifier prefix $\exists f\exists g\forall x\forall y\forall z$, because not all quantifier-free formulas $\epsilon(x,y,z)$ are equivalent to quantifier-free formulas of the form $\epsilon_1(x,y,f(y)/u) \lor \epsilon_2(x,z,g(z)/v)$ (i.e., formulas which have $\lor$ as their most external operator; in which each disjunct is a two-variable formula; 
in which all occurrences of $f$ are restricted to be applied to $y$, and all occurrences of $g$ are applied to $x$). Thus, we will either have to conjure new descriptions of NP-complete problems, or instead make use of the semantics and the inferential rules for IF logic in order to show that the given fragments are first-order.

Summarizing, in the present paper we give a partial classification of the complexity of the fragments of regular $IF$ logic which are induced by tree prefixes; some of these fragments are shown to contain descriptions of NP-complete problems, while others are shown to express only first-order concepts (more precisely, each sentence of such a fragment is shown to be equivalent to a first-order sentence). 
In section \ref{SEM} we briefly present $IF$ logic.  Section \ref{EQRULES}  presents some notions of equivalence of $IF$ sentences and formulas and reviews several equivalence rules of $IF$ logic that are used throughout the paper. An earlier draft (\cite{Bar2016}) made use of syntactical manipulations of trees rather than formulas; the difficulties of this alternative approach are succinctly described in appendix \ref{APPMANTREE}. Section \ref{EQUIVTREESSEC1} introduces syntactical trees and related notions; section \ref{COMPLEXITY1} isolates a notion of complexity for tree prefixes. In section \ref{SECROUGH}, we extend the notions of Henkin and signalling prefixes to the case of trees; then two new significant classes of tree prefixes are introduced: the \emph{generalized Henkin} and the \emph{coordinated} ones. It is then shown that the search for NP-hard prefixes can be narrowed down to the Henkin, signalling, generalized Henkin and coordinated classes; instead, sentences that  have as prefix one of the remaining trees (called \emph{modest} trees) can be mechanically transformed into first-order sentences. In section \ref{HIGHCOMP} we generalize Sevenster's extension lemma, showing that taking extensions of regular syntactical trees preserves properties such as NL,P,NP-hardness; and we use it to show that all trees that contain Henkin or signalling patterns are NP-complete. Section \ref{GHTREES} divides the generalized Henkin fragment into four subclasses; of these, one is shown to contain only NP-hard prefixes, which can express the SAT problem; for the other three classes, we give partial results, showing that many of the trees they contain are in FO. We also account for the recent discovery (\cite{JaaBar2019}) of a new NP-hard tree. Section \ref{CTREES} considers a first kind of coordinated trees, which are classified into three subclasses, all shown to be NP-complete (the first two define SAT, and the third one the SET SPLITTING problem). We also show that the trees in the third class can express 2-COLORABILITY (a logspace, non first-order problem). Section \ref{CTREES2} takes briefly into account the remaining coordinated trees (of ``second kind''), showing that a few of them are in FO. Two difficult proofs are postponed to appendices \ref{APPMODEST} and \ref{APPEXTLEM} for the sake of  readability.

  \section{$IF$ logic} \label{SEM}

We present here the syntax of $IF$ logic in the form which is also sometimes called \emph{slash logic}. The reader can consult \cite{ManSanSev2011} for further details on this language. We assume a countable set of (individual) variables. Signatures and terms are defined as for first-order logic.  An \textbf{$IF$ formula} is an expression of one of the following forms
$$t_1 = t_2 \ | \ t_1 \neq t_2 \ | \ R(t_1,\dots,t_n) \ | \ \neg R(t_1,\dots,t_n) \ | \ \psi \land \chi \ | \ \psi \lor \chi \ | \ (\forall v/V)\psi \ | \ (\exists v/V)\psi$$
where $t_1,\dots,t_n$ are terms, $R$ is an $n$-ary relation symbol, $v$ a variable, $V$ a finite set of variables (called \emph{slash set}), $\psi$ and $\chi$ $IF$ formulas.\footnote{Variants of this language allow a so-called \emph{dual negation} to occur in front of any (sub)formula. For most purposes, however, this extended language offers nothing new with respect to ours, in which formulas are negation normal -- only atomic formulas can be negated. See \cite{ManSanSev2011} for a discussion.} Formulas of the forms $t_1 = t_2,  t_1 \neq t_2,  R(t_1,\dots,t_n), \neg R(t_1,\dots,t_n)$ are called, as usual, \textbf{literals}. 

The set of free variables of a given $IF$ formula is defined inductively as follows:
\begin{itemize}
\item $FV(\psi\land\chi) = FV(\psi\lor\chi) = FV(\psi)\cup FV(\chi)$
\item $FV((\forall v/V)\psi) =FV((\exists v/V)\psi) =(FV(\psi)\setminus \{v\})\cup V$
\end{itemize}
Thus, also occurrences of variables in slash sets may be counted as free. 
 The set of bound variables of an $IF$ formula $\psi$, denoted as $Bound(\psi)$, is as usual the set of variables that occur quantified in $\psi$.
 If $FV(\psi)=\emptyset$, then $\psi$ is said to be an \textbf{$IF$ sentence}; otherwise, it is an \textbf{open formula}. If  $Bound(\psi)=\emptyset$, then $\psi$ is said to be \textbf{quantifier-free}.

For brevity, we will sometimes write quantifiers as $(Qv/u_1,\dots,u_n)$ instead of $(Qv/\{u_1,\dots,u_n\})$. A quantifier with empty slash set, say $(Qv/\emptyset)$, will be simply written $Qv$. A conservativity result (theorem \ref{TEOCONS} below) supports the identification of quantifiers with empty slash set with first-order quantifiers. 

In the introduction, we made use of games in order to give an idea of the meaning of $IF$ sentences. For technical purposes, it will be convenient to use in the rest of the paper a different semantics, nowadays called \emph{team semantics} (\cite{Hod97}, \cite{Vaa2007}), which is in accordance with the game-theoretical account over sentences, but also assigns a meaning to open formulas. 
In team semantics, formulas are interpreted over sets of assignments of a common variable domain (\emph{teams}), and thus their ``meanings'' are sets of teams.\footnote{The word ``team'' as used in this context has no relation to the teams of players from the game-theoretical interpretation.} Indeed, intuitively the notion of independence has no meaning over single assignments, and this intuition has been assessed by a combinatorial argument (\cite{CamHod2001}). We will write $M,X\models\varphi$ to say that the formula $\varphi$ is satisfied by the team $X$ on a (first-order) structure $M$.

\begin{conv}
By a structure $M$ we mean a pair $(dom(M),I_M)$, where $dom(M)$ is a set and $I_M$ is a function that maps each element of the signature into its interpretation (defined as for first-order logic). As is common, we write $M$ for $dom(M)$ when there is no risk of ambiguity.
\end{conv}

\begin{df}
A \textbf{team} $X$ on a structure $M$ is a set of assignments such that, for all $s,s'\in X$, $dom(s)$ is a finite set of variables, and $dom(s) = dom(s') = : dom(X)$.  

A team $X$ is \textbf{suitable} for a formula $\psi$ in case $FV(\psi)\subseteq dom(X)$.

A structure $M$ is \textbf{suitable} for $\psi$ if the signature of $M$ contains all the nonlogical symbols of $\psi$.
\end{df}

\begin{df}
Given a team $X$ over a structure $M$ and a variable $v$, the \textbf{duplicated team} $X[M/v]$ is defined as the team $\{s(a/v) \ | \ s\in X,a\in M\}$.

Given a team $X$ over a structure $M$, a variable $v$ and a function $F: X \rightarrow M$, the \textbf{supplemented team} $X[F/v]$ is defined as the team $\{s(F(s)/v) \ | \ s\in X\}$. 
\end{df}

\begin{df}
Given two assignments $s,s'$ with the same domain, and a set of variables $V$, we say that $s$ and $s'$ are \textbf{$V$-equivalent}, and we write $s\sim_V s'$, if $s(x) = s'(x)$ for all variables $x\in dom(s)\setminus V$.\\ 
Given a team $X$, a structure $M$ and a set $V$ of variables, a function $F:X\rightarrow M$ is \textbf{V-uniform} if $s\sim_V s'$ implies $F(s)=F(s')$ for all $s,s'\in X$.  
\end{df}

The notion of satisfaction by a team is defined by the following compositional clauses, which we present in the style of \cite{Vaa2007}. We assume familiarity with Tarskian semantics.

\begin{df}
We say that a suitable team X satisfies an  IF formula $\varphi$ over a structure $M$, and we write $M,X\models \varphi$  if any of the  following holds:
\begin{itemize}
\item For $\alpha$ literal, $M,X\models \alpha$ if $M,s\models \alpha$ in the Tarskian sense for every $s\in X$.
\item $M,X\models \psi\land\chi$ if $M,X\models\psi$ and $M,X\models\chi$.  
\item $M,X\models \psi\lor\chi$ if there are $Y,Z\subseteq X$ such that $Y\cup Z=X$, $M,Y\models \psi$, and $M,Z\models \chi$.
\item $M,X\models (\forall v/V)\psi$ if $M,X[M/v]\models \psi$.
\item $M,X\models (\exists v/V)\psi$ if $M,X[F/v]\models \psi$ for some $V$-uniform function $F:X\rightarrow M$.
\end{itemize}
\end{df}

\begin{df}
An $IF$ sentence $\varphi$ is said to be \textbf{true} in a structure $M$, and we write $M\models\varphi$, if $M,\{\emptyset\}\models\varphi$.\footnote{Here $\{\emptyset\}$ denotes the singleton team containing the empty assignment.}
\end{df}

\noindent These definitions conservatively extend the usual semantics of first-order logic, in the following sense:

\begin{prop}\label{TEOCONS}
1) (\cite{CaiDecJan2009}, Theorem 4.11)  Let $\varphi$ be an $IF$ sentence which is syntactically first-order (i.e., all of its slash sets are empty). Then $M\models\varphi$ according to team semantics if and only if $M\models\varphi$ according to Tarskian semantics.

2) (\cite{CaiDecJan2009}, Lemma 4.10) Let $\psi$ be an $IF$ formula which is syntactically first-order.  Then $M,X\models\psi$ according to team semantics if and only if, for all $s\in X$, $M,s\models\psi$ according to Tarskian semantics.
\end{prop}


The focus of the paper will be on the \emph{regular} fragment of $IF$ logic, in which requantification is forbidden:

\begin{df}
An $IF$ formula $\psi$ is said to be \textbf{regular} if:
\begin{enumerate}
\item Variables are not requantified, i.e., if a quantifier $(Qv/V)$ occurs in $\psi$, then no other quantifier of the form $(Q'v/V')$ occurs in the scope of  $(Qv/V)$.
\item No variable occurs both free and bound in $\psi$.
\end{enumerate}
\end{df}

\noindent Notice that condition 2. is automatically satisfied by sentences.

\section{Equivalence of sentences and formulas} \label{EQRULES}

\begin{df}
Two $IF$ sentences are (truth-)equivalent if they are true in the same structures (i.e., $\varphi\equiv\chi$ if for all structures $M$, $M\models\varphi\Leftrightarrow M\models\chi$). 
\end{df}



We will need a well known fact about the expressivity of $IF$ sentences:

\begin{prop} \label{FAGIN} (\cite{ManSanSev2011}, Theorems 6.10, 6.16)
On the sentence level, $IF$ logic is equiexpressive with existential second order logic. Thus, by Fagin's theorem (\cite{Fag1974}), the set of $IF$ sentences characterizes the complexity class NP.
\end{prop}

A richness of equivalence rules for $IF$ formulas was developed (mainly) in \cite{Dec2005}, \cite{CaiDecJan2009}, \cite{Man2009}, \cite{ManSanSev2011}, \cite{Bar2013}, \cite{Sev2014}; we list here those rules that will be needed in the following. These rules act on formulas, so our notion of truth-equivalence of sentences does not suffice to describe them. Many alternatives have been considered in the literature for what regards equivalence of $IF$ formulas; 
  the simplest option would be to consider two formulas $\psi,\theta$ equivalent if in all structures they are satisfied by the same teams, \emph{provided that we only consider teams whose variable domain contains $FV(\psi)\cup FV(\theta)$}. However, many important equivalence rules of $IF$ logic are context-dependent: they hold only if some kinds of restrictions are imposed on the contexts in which the formulas may appear; that is, these rules only hold if the formulas do not occur in the scope of certain quantifiers. Thus, it is in many occasions more convenient to consider notions of equivalence relativized to contexts. 
We do it here in the style of Caicedo, Dechesne and Janssen (\cite{CaiDecJan2009}), specifying which variables should not appear in the context.\footnote{Actually, our definitions will slightly differ from those of \cite{CaiDecJan2009}. Our definitions are meant to extend truth-equivalence to open formulas,  \cite{CaiDecJan2009} aimed at extending a stricter notion called ``strong equivalence''.}  


\begin{df}
Let $\psi$ be an $IF$ formula, $Z$ a finite set of variables. Then $\psi$ is \textbf{$Z$-closed} if $FV(\psi)\cap Z = \emptyset$.
\end{df}

\begin{df}
Let $\psi,\chi$ be $IF$ formulas, let $Z$ be a finite set of variables. We say that $\psi$ and $\chi$ are \textbf{$Z$-equivalent}, and we write $\psi\equiv_Z\chi$, if they are $Z$-closed and, furthermore, $M,X\models \psi \Leftrightarrow M,X\models \chi$ for all structures $M$ and for all teams $X$ that are suitable for $\psi$ and $\chi$ and such that $dom(X)\cap Z = \emptyset$. 

If we have an explicit listing $\{z_1,z_2,\dots,z_n\}$ of $Z$, we can also write, for brevity, $\psi\equiv_{z_1z_2\dots z_n}\chi$.
\end{df}

So, the subscripts to the equivalence symbols mean that the equivalence only holds for those teams whose domain does not contain any of the subscripted variables; and also, in order to avoid triviality, the subscripted variables must not occur free in the formulas under consideration. This notion of equivalence of formulas works well because of the following two facts:

\begin{prop} \label{TEOEQSENT} 

 (\cite{CaiDecJan2009}, remarks on page 22)
If $\varphi$ and $\chi$ are $IF$ sentences, then, for any finite set $Z$ of variables, $\varphi\equiv_Z \chi$ if and only if $\varphi\equiv \chi$.


\end{prop}

\begin{prop} \label{TEOSUBEQ} (\cite{CaiDecJan2009}, Theorem 6.14)
 If $\varphi,\psi,\psi'$ are $IF$ formulas, $Z$ a finite set of variables, $\varphi'$ is obtained from $\varphi$ by replacing a subformula occurrence of $\psi$ with $\psi'$, and $\psi\equiv_Z\psi'$, then $\varphi\equiv_Z\varphi'$.
\end{prop}

We can now present the equivalence rules that we shall need. We just point out that many of the context restrictions of each rule can in practice be ignored when applying the rules within a regular formula.

\begin{prop}[Renaming]\label{VARIANT}\footnote{A note of warning. This proposition is nothing else than Theorem 6.12 of \cite{CaiDecJan2009}. If the reader compares our formulation with the rule stated in that paper, (s)he might think that we have forgotten a clause; that we should have specified that $u$ must not be in $U$. Yet, this is already implied by $uv$-closedneess: the formulation in \cite{CaiDecJan2009} was redundant.}
Suppose $u$ is not bound in $\psi$. If $v$ does not occur in $(Qu/U)\psi$, then
\[
(Qu/U)\psi \equiv_{uv} (Qv/U) Subst(\psi,u,v).  
\]
where $Subst(\psi,u,v)$ is the formula obtained by replacing, in $\psi$, all free occurrences of $u$ with $v$.
\end{prop}


When extracting a quantifier $(Qu/U)$, say, from a left disjunct (resp.conjunct), the variable $u$ must in general be added to the slash sets of the right disjunct  in order to prevent it to be used as a source of signals -- which could be used to circumvent the restrictions imposed by slash sets (see \cite{CaiDecJan2009}). However, using the form of extraction rule that we review below, we can avoid adding $u$ to \emph{empty} slash sets (i.e., we can preserve the first-order quantifiers).

\begin{df} \label{VERTICALSLASH}
Given an IF formula $\psi$, we define $\psi|_v$ to be the formula obtained by adding the variable $v$ to all \emph{nonempty} slash sets of $\psi$; and similarly for syntactical trees.
\end{df}

\begin{prop}[Strong quantifier extraction, a special case of Theorem 8.3 of \cite{CaiDecJan2009}]\label{STRONGEXTRACTION}
If $u$ does not occur in $\psi$ nor $U$, then:
\[
    (Qu/U)\varphi \circ \psi \equiv_u (Qu/U)(\varphi \circ \psi|_u)
\]
where $\circ$ is either $\land$ or $\lor$.
\end{prop}

We list two more useful equivalence rules, distribution of universal quantifiers (see \cite{ManSanSev2011}, 5.23) and quantifier swapping (\cite{ManSanSev2011}):

\begin{prop}[Distribution of universal quantifiers over conjunctions] \label{DISTRIBUTION}
For all $\varphi,\psi$ $IF$ formulas:
\[
   \hspace{10pt} \forall u(\varphi \land \psi) \equiv_u \forall u\varphi \land \forall u\psi.
\] 
\end{prop}   

\begin{prop}[Quantifier swapping]\label{SWAP}
Let $Q,Q'$ be quantifiers, $\psi$ an $IF$ formula. Then:
\[
   (Qu/U)(Q'v/V\cup\{u\})\psi \equiv_{uv} (Q'v/V)(Qu/U\cup\{v\})\psi.
\]
\end{prop}

\noindent Observe that adjacent quantifiers of the same kind are not always allowed to commute: for example, notice that in the left member of the above formula we require $u$ to occur in the slash set of $v$. It is also worth noting that the usual first-order rule for swapping quantifiers of the same type is not a special case of this scheme; indeed, $IF$ logic has a second rule which allows swapping certain quantifiers of the same type (\cite{CaiDecJan2009}, Theorem 13.3); we will only make use of the following very special case:

\begin{prop}[Swapping first-order universal quantifiers]\label{UNISWAP}
For any $IF$ formula $\psi$: 
\[
\forall u \forall v\psi \equiv \forall v \forall u\psi. 
\]
\end{prop}

\noindent Finally, we look at two rules which are specific of $IF$ quantification.\footnote{The previous rules hold under a more restrictive notion called \emph{strong equivalence}, under which $IF$ logic is treated as a three-valued logic. The two remaining rules, instead, owe their validity to the fact that we are only analyzing the truth of sentences, not their falsity.}


\begin{prop}[Slash sets of universal quantifiers are irrelevant]  \label{UNIDEP}
For any $IF$ formula $\psi$, 
\[
(\forall u/U)\psi \equiv \forall u \psi.
\]
\end{prop}

\begin{prop}[Purely existential slash sets are irrelevant]  \label{DEPEX}
Suppose two regular IF sentences $\varphi,\varphi'$ differ only for one quantifier, which is $(\exists v/V)$ in $\varphi$ and $\exists v$ in $\varphi'$; suppose furthermore that all variables in $V$ are existentially quantified.  Then $\varphi\equiv\varphi'$.
\end{prop}

   \section{Syntactical trees: basic definitions} \label{EQUIVTREESSEC1}

We define here the class of syntactical trees which is of our interest -- we are seeking for the simplest possible generalization of what a prefix is if we do not restrict attention to prenex sentences. This requires including in the prefixes also connectives, and taking into account the binary ramifications they induce in the structure of formulas. This class of trees (the \emph{positive initial} trees) has already been introduced elsewhere (\cite{Bar2013}), but here we will require some more precision in the formal details. For technical ease, in our trees we will allow occurrences of the gap symbol $[\phantom{a}]$. Each gap symbol is a marker for a node to which (the tree of) some $IF$ formula might potentially be attached. By a \textbf{tree} here we mean a finite partially ordered set $(T,\preceq_T)$ with a minimum element (the \textbf{root}) and such that, for each $t\in T$, the set of predecessors of $t$ is linearly ordered by $\preceq_T$. 

\begin{df}
A \textbf{syntactical tree} is a (finite) tree whose nodes are occurrences of atomic formulas, negation, conjunction, disjunction, quantifiers (with their slash sets), and the gap symbol $[\phantom{a}]$, and which respects the following constraints: 
1) atomic formulas are leaves (i.e., they have no successors) \\
2) gap symbols are leaves \\
3) each negation has exactly one successor \\
4) each binary connective has exactly two successors \\
5) each quantifier has exactly one successor. 
\end{df}
It should be clear in what sense to each $IF$ sentence we can associate \emph{its} syntactical tree (which is, of course, a tree without gaps), and in the following we will always indentify a formula with its tree. 
Here are some examples of trees that are not the syntactical tree of any formula, since they contain gaps:

\vspace{5pt}

\Tree  [.$\forall x$  [.$\exists y$ [.$[\phantom{a}]$ ] ]  ] 
\Tree  [.$\forall x$  [.$\lor$ [.$[\phantom{a}]$ ] [.$\gap$ ] ] ]   
\Tree  [.$(\forall x/\{z\})$  [.$\lor$ $A(x)$ [.$\neg$ [.$[\phantom{a}]$ ] ] ]  ] 
\Tree  [.$\forall x$  [.$\lor$ [.$[\phantom{a}]$ ] [.$\land$ $B(y)$ $C(z)$ ] ] ]   
\Tree  [.$\lor$  [.$\exists x$ $A(x)$ ]  [.$\exists y$ [.$[\phantom{a}]$ ]  ] ]   

\vspace{5pt}

We will use some terminology which is standard for trees:

\begin{df}\label{DEFCHAINBRANCH}
If  $(T,\preceq_T)$ is a syntactical tree, a \textbf{chain} of $(T,\preceq_T)$ is  a pair $(S, \preceq_S)$, where $S\subseteq T$, $\preceq_S$ is the restriction of $ \preceq_T$ to $S$, and  $\preceq_S$ linearly orders $S$. 

A \textbf{branch} of $(T,\preceq_T)$ is a maximal chain of $(T,\preceq_T)$.\footnote{I.e., a chain $(S, \preceq_S)$ such that, for each $t\in T\setminus S$, the set $S\cup\{t\}$, together with the restriction of $\preceq_T$ to $S\cup\{t\}$, is not a chain of $(T,\preceq_T)$ ).} 
\end{df}

The notion of a quantifier prefix is generalized by the following class of syntactical trees:

\begin{df}
A \textbf{positive initial tree}, or \textbf{tree prefix}, is a syntactical tree which contains no occurrence of atomic formulas nor of negation.
\end{df}
Said otherwise, a positive initial tree can be obtained from the syntactical tree of some negation normal $IF$ formula by removing from it all nodes that correspond to literals.
 The word \emph{positive} refers to the fact that we do not allow negation symbols to occur in the tree, while the word  \emph{initial} refers to the fact that none of the branches of the tree end with an atomic formula 
(among the trees in the previous picture, only the first and second are positive and initial; the fourth and fifth are positive but not initial). This generalizes the fact that a quantifier prefix is an initial segment of the syntactical tree of a formula; the obvious generalization of ``initial segment'' for a tree is the notion of \emph{down set}.

\begin{df}
A \textbf{down set} $Y$ of a tree $T$ is $Y\subseteq T$ such that
\[
\forall y\in Y\forall t\in T(t\preceq_T y \rightarrow t\in Y).
\] 
\end{df}
\begin{df}
Let $T$ be a syntactical tree, and $T^-$ the tree obtained by removing the gap nodes from $T$. An IF formula $\varphi$ \textbf{begins with $T$} if $T^-$ is a down set of the syntactical tree of $\varphi$.
\end{df}
\begin{df}
Given an $IF$ formula $\varphi$, we define \textbf{the tree prefix of $\varphi$}, and denote it as $PTr(\varphi)$, to be the largest positive initial tree $T$ such that $\varphi$ begins with $T$.
\end{df}


For example, we can say that the formula $\varphi = \forall x(A(x)\lor \neg B(x))$ begins with the tree

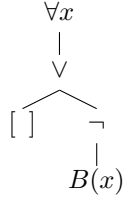
\vspace{5pt}

\Tree  [.$\forall x$  [.$\lor$ $[$\phantom{a}$]$ [.$\neg$ [.$B(x)$ ] ] ]  ]  \\

\vspace{5pt}

\noindent even though this tree is not positive initial (it contains a negation, and also an atomic formula). Instead, $PTr(\varphi)$ is

\vspace{5pt}

\Tree  [.$\forall x$  [.$\lor$ $[$\phantom{a}$]$ [.$[$\phantom{a}$]$ ] ]  ] 

\vspace{5pt}

\noindent We can also write these kinds of trees in linear notation; e.g., the tree above is $\forall x([\phantom{a}]\lor[\phantom{a}])$.




The notion of \emph{subtree} is in a sense dual to the notion of a prefix.

\begin{df}
An \textbf{up set} $Y$ of a tree $T$ is $Y\subseteq T$ such that
\[
\forall y\in Y\forall t\in T(y\preceq_T t \rightarrow t\in Y).
\] 
\end{df}

\begin{df}
A \textbf{subtree} $S$ of a tree $T$ is a suborder of $T$ which 1) is an up set of $T$, and 2) has a root (i.e. a minimum according to $\preceq_T$). 
\end{df}

So, a subtree of $T$ is obtained whenever we choose a node $t$ of $T$ and we pick all nodes that follow $t$ in the ordering $\preceq_T$. We might also say that such a subtree is made of $t$ and of all the nodes of $T$ which are in the scope of $t$. 

\begin{notation}
Given a formula $\varphi$ (resp. a quantifier prefix $\vec Q$, a syntactical tree $T$), we denote the relation of scope between pairs of logical operators as $\prec_\varphi$ (resp. $\prec_{\vec Q}$, $\prec_T$). So, for example, $\forall x \prec_\varphi \exists y$ means that (a specific occurrence of) $\exists y$ occurs within the scope of (a specific occurrence of) $\forall x$ in formula $\varphi$. In the case of trees,  $\prec_T$ is just the strict partial order which is associated to the ordering $\preceq_T$ of the tree.
\end{notation}

Two quantifier prefixes $R$, $S$ can obviously always be concatenated in order to obtain a longer prefix $RS$; this notation can sometimes be extended to trees:

\begin{notation}
Whenever $R$ is a finite linearly ordered syntactical tree whose last element is a gap, and $S$ is a tree, we can unambiguosly denote as $RS$ the concatenation of $R$ and $S$, that is, the tree obtained by removing the last (gap) node of $R$ and replacing it with the tree $S$.
\end{notation}

\begin{df}
An \textbf{incomplete branch} of $T$ is a branch whose last element is a gap $[\phantom{a}]$. The set of the incomplete branches of $T$ will be denoted as $IBranch(T)$. 
\end{df}


\noindent So, there is an obvious bijection between incomplete branches of a tree $T$, and the occurrences of gap symbols in $T$.

The main results of \cite{Sev2014} worked properly only for a restricted class of quantifier prefixes, the \emph{sentential} class. These are prefixes that can in principle give rise to a sentence when they are prefixed to some appropriate quantifier-free formulas. An example of a \emph{non}-sentential prefix is $\forall x(\exists y/z)$. Any formula which begins with this prefix is not a sentence, because it has a free variable, $z$. We define here a class of tree prefixes which incorporates both the notions of sententiality and \emph{regularity} (which mainly amounts to forbidding requantification of a variable). For brevity (and since sententiality is an obvious requirement for the kind of analysis we pursue here) such trees will be simply called ``regular''.


The definitions of the sets of  free and bound variables, which we have given for formulas, extend in a straightforward way to trees, if one reads $(Qu/U)\prec_T(Qv/V)$ as ``$(Qv/V)$ is in the scope of $(Qu/U)$''. We write $FV(T)$, resp. $Bound(T)$ for these sets. 


\begin{df} \label{REGU2}
A tree prefix $T$ is \textbf{regular} if the following hold:\\
0) If a quantifier $(Qv/V)$ occurs in $T$, and a variable $u$ is in $V$, then there is in $T$ another quantifier $(Q'u/U)\prec_T(Qv/V)$ (sententiality).\\
1) No variable occurs both free and bound in $T$.\\
2) If a quantifier $(Qv/V)$ occurs in $T$, then it is not in the scope of any quantifier of the form $(Qv/W)$. 
\end{df}
 This definition slightly clashes with our earlier definition of a regular $IF$ formula, which did not include requirement $0$). However, in this paper we deal almost exclusively with sentences; and notice that sentences automatically satisfy conditions 0) and 1).

Now we specify what $IF$ sentences can be obtained by filling the incomplete branches of a tree with (the trees of) quantifier-free formulas. The first of the following definitions generalizes the operation of postfixing an open formula to a quantifier prefix;  here we may need to attach \emph{many} formulas, one for each gap in the tree.

\begin{df} \label{DEFCF}
Let $T$ be a syntactical tree.  We call any function $e:IBranch(T)\rightarrow QFree$ a \textbf{completing function for $T$}.


 A completing function is \textbf{sentential} if, for each $R\in IBranch(T)$, we have $FV(e(R))\subseteq Bound(R)$. 


\end{df}

\begin{df}
For any tree $T$ and completing function $e$ (for $T$), we call $\hat e(T)$ the formula obtained replacing, for each $R\in IBranch(T)$,  the gap at the of $R$ with  $e(R)$. We will call the formula $\hat e(T)$ a \textbf{completion} of $T$.

If $S$ is a subtree of $T$, we denote by $\hat e(S)$ the smallest subformula of $\hat e(T)$ which contains $S$.
\end{df}

It should be clear that, if $T$ is a regular tree prefix, then asserting that $e:IBranch(T)\rightarrow QFree$ is sentential amounts to saying that $\hat e(T)$ is a sentence.

\begin{example}
1) The simplest possible examples of tree prefixes are the quantifier prefixes. For instance, let $T = \exists y(\forall x/\{y\})[\phantom{a}]$. There is only one gap, so only one incomplete branch, which is $T$ itself (with its ordering); so, a completing function for $T$ is just a function from the singleton set $\{T\}$ to $QFree$. 
Set for example $e(T) = P(x)\land Q(x,y)$. Applying this completing function to $T$, one obtains the formula $\hat e(T) = \exists y(\forall x/\{y\})(P(x)\land Q(x,y))$. Notice that we have $FV(e(T)) = \{x,y\} = Bound(\hat T)$: $e$ is sentential, and indeed $\hat e(T)$ is a sentence. 





2) Consider a tree of the form $T' = \exists y([\phantom{a}]\lor(\forall x/\{y\})[\phantom{a}])$, which is not linear. It has two branches. Call $A$ the branch  containing the leftmost gap, and $B$ the other one. A completing function for $T'$ will be a function $j:\{A,B\}\rightarrow QFree$, for example
\begin{equation*}
\left\{\begin{array}{l}
j(A)=P(y,z)\\
j(B)=Q(x,y)
\end{array}
\right.
\end{equation*}

\noindent which is not sentential because of $z$ occurring free in $j(A)$. The result of the completion is the (open) formula $\hat j(T') = \exists y(P(y,z)\lor (\forall x/\{y\})Q(x,y))$ with free variable $z$.

If we instead define a completing function $k$ by $k(A):= S(y), k(B):=Q(x,y)$, then $k$ is sentential, and $\hat k(T') = \exists y(S(y) \lor (\forall x/\{y\})Q(x,y))$ is a sentence. 
\end{example}

  \section{Complexity of $IF$ tree prefixes} \label{COMPLEXITY1}

We assume the reader is familiar with basic notions of complexity theory, in particular reductions, hardness, completeness and the complexity classes  FO, L, NL, P and NP. In the following, when we speak of NP-completeness, we are thinking of completeness up to polynomial reductions (although, in most cases, much weaker reductions are adequate). It is known that the following inclusions hold:
\[
\text{FO}\subseteq\text{AC}^0\subset\text{TC}^0\subseteq\text{L}\subseteq\text{NL}\subseteq\text{P}\subseteq\text{NP}
\] 
where AC$^0$ and TC$^0$ are two classes of computation by circuits (AC$^0$: problems decidable by boolean circuits of unbounded fan-in and constant depth, TC$^0$: problems solvable by threshold circuits of constant depth). $\text{AC}^0\subset\text{TC}^0$ is one of the few strict inclusions that are known of within NP; it has the interesting consequence that first-order formulas cannot even express all L problems.

We study the complexity of $IF$  positive initial trees, in the sense given by the following definitions (given along the lines of \cite{BlaGur1986}).

\begin{df} \label{DEFC}
To each $IF$ sentence $\varphi$, we associate the class $F_\varphi = \{M \ | \ M \text{ finite, }$ $M\models \varphi\}$ of its finite models.
Given a syntactical tree $T$, we define the \textbf {complexity class of $T$}:
\begin{equation*}
\text{\emph C}(T) = \{F_{\hat e(T)} \ | \text{ } e \text{ is a sentential completing function for } T\}.
\end{equation*}
 If we have \emph{C}$(T) =$ \emph{C}$(T')$, resp. \emph{C}$(T) \subseteq$\emph{ C}$(T')$... then we say that $T$ is as complex as $T'$, resp. $T$ is less complex than $T'$... 
\end{df}

The complexity classes $C(T)$, as defined above, are, from the set-theoretical point of view, metaclasses. If the reader is worried by this point, (s)he just has to replace the above definition of $F_\varphi$  with $\{[M] \ | \  M \text{ finite, }M\models \varphi\}$, where $[M]$ is a fixed representative of the isomorphism class of $M$, having as domain a subset of $\mathbb{N}$. All finite models have isomorphic copies of this kind, and $\{[M] \ | \ M \text{ finite, }M\models \varphi\}$ is a set.

\begin{df} 
The (``data complexity'' version of) the model-checking problem for an $IF$ sentence $\varphi$ is the problem of establishing whether $M\models \varphi$ when a (representation of a) finite structure $M$ is given as input. 
\end{df}



\begin{df}
  We shall say that a regular tree prefix $T$ \textbf{is in complexity class \emph{K}} if for all sentential completing functions $e$ the model-checking problem for $\hat e(T)$ is in \emph{K} (equivalently: if $\text{\emph C(T)}\subseteq \text{\emph K}$).

We say that $T$ is \textbf{\emph{K}-hard}, or that it \textbf{encodes a \emph{K}-hard problem}, if there is at least one sentential completing function $e$ such that the model-checking problem for $\hat e(T)$ is \emph{K-hard} (equivalently: if $\text{\emph C(T)}\cap \text{\emph K-hard} \neq \emptyset$).

If $T$ is in \emph{K} and it is \emph{K}-hard, we say it is \textbf{\emph{K}-complete}\footnote{We might say, more properly, K-complete \emph{up to reduction closure}. Even if a tree prefix $T$ is K-complete in this sense, there may be problems from K which are not definable by sentences that begin with $T$. \cite{BlaGur1986} also use the term \emph{mighty} to refer to an NP-complete prefix.} (equivalently: if $\text{\emph C(T)}\cap \text{\emph K-complete} \neq \emptyset$ \emph{and} $\text{\emph C(T)}\subseteq \text{\emph K}$).
\end{df}


\section{A rough classification of tree prefixes} \label{SECROUGH}

We are now in the condition to enunciate in our framework the dichotomy result given by Sevenster (\cite{Sev2014}, Theorem 5.1), restricted to the case of $IF$ regular prefixes:
\begin{prop} \label{OLDDICHOTOMY}
Every regular $IF$ quantifier prefix either encodes an \emph{NP}-complete problem, or it is in the class $\operatorname{FO}$ of first-order definable problems.
\end{prop}

This result can be stated in a stronger form, saying 1) that the FO prefixes are equivalent, in a rather strong sense, to syntactically first-order prefixes, and 2) giving a complete (and effective) classification of the NP-complete vs. the FO prefixes. The NP-complete prefixes were classified according to the presence of particular patterns of dependence and independence among the quantifiers. We define analogous classes for syntactical trees. 

\begin{df}
We say that a quantifier $(Qy/Y)$, occurring in a regular formula or a regular tree, \textbf{depends on} $(Q'x/X)$ if $(Q'x/X)\prec(Qy/Y)$ and $x\notin Y$. If any of these two conditions does not hold, we say that $(Qy/Y)$ \textbf{does not depend on} $(Q'x/X)$.

For brevity, we will sometimes more simply write that $y$ depends (resp. does not depend) on $x$.
\end{df}

We define two ``branch properties'' which generalize the homonymous properties defined in \cite{Sev2014}. They identify branches which mimick  Henkin quantifiers (\cite{Hen1961}) and branches which contain signalling patterns (\cite{Hod97}, \cite{Jan2002}).

\begin{df}
A branch of a syntactical tree is \textbf{Henkin} if it contains quantifiers $(\forall x/X),(\exists y/Y),(\forall z/Z),(\exists w/W)$ such that:\\
1) $(\exists y/Y)$ depends on $(\forall x/X)$ but does not depend on $(\forall z/Z)$ nor $(\exists w/W)$\\
2) $(\exists w/W)$ depends on $(\forall z/Z)$ but does not depend on $(\forall x/X)$ nor $(\exists y/Y)$.
\end{df}
 
\begin{df}
A branch of a syntactical tree is \textbf{signalling} if it contains quantifiers $(\forall x/X),(\exists y/Y),(\exists z/Z)$ such that:\\
1) $(\exists y/Y)$ depends on $(\forall x/X)$ \\
2) $(\exists z/Z)$ depends on $(\exists y/Y)$ but does not depend on $(\forall x/X)$.
\end{df}

\begin{df}
Given any property $\mathbb{P}$ of branches, we say that a tree $T$ has property $\mathbb{P}$ if there is a branch of $T$ which has property $\mathbb{P}$.
\end{df}

Of course, not every interesting property of a syntactical tree is induced in this way from the properties of its branches; the next two properties of trees exemplify this point.

\begin{df}
A syntactical tree T is \textbf{first-order} if all of its slash sets are empty.
\end{df}

\begin{df}
A syntactical tree T is \textbf{primary} if it is neither Henkin nor signalling.
\end{df}

With these definitions, the classification result of Sevenster can be summarized more precisely thus:
\begin{prop} \label{SEVCLASSIFY} (\cite{Sev2014}) \\
1) Henkin and signalling (regular) $IF$ quantifier prefixes are \emph{NP}-complete. \\
2) Primary regular $IF$ quantifier prefixes are in $\operatorname{FO}$.
\end{prop}

The apparatus developed so far allows us to (partially) extend the result on primary prefixes to trees. We define three new classes of trees. We will show that the first two of these, the \emph{generalized Henkin} and the \emph{coordinated} class, delimit the space for searching for genuinely new NP-complete prefixes.

\begin{df}
A syntactical tree $T$ is \textbf{generalized Henkin} if it contains logical operators $\forall x,\forall y,\circ,(\exists u/U),(\exists v/V)$  ($\circ$ being either $\land$ or $\lor$) such that:\\
0) $\forall x,\forall y,(\exists u/U),(\exists v/V)$  do not all occur in a same branch\\
1) $\forall x\prec_T \circ \prec_T (\exists u/U),(\exists v/V)$\\
2) $u$ depends on $x$ but does not depend on $y$ nor $v$\\
3) $v$ depends on $y$ but does not depend on $x$ nor $u$. 
\end{df} 

\noindent We can think of this structure as a sort of Henkin prefix which is split over multiple branches of the tree. 
  A typical example of a generalized Henkin tree prefix is  $\forall x(\exists u[\phantom{a} ]\lor(\forall y(\exists v/\{x\})[\phantom{a}])$, in which the two existentials occur in distinct branches of the tree. This is a syntactical structure that can support phenomena of ``signalling by disjuction''. But notice that, in the general definition, $\circ$ is also allowed to be a conjunction. Clause 0) excludes the Henkin tree prefixes from this class, while clause 1) excludes some first-order trees, such as $(\forall x\exists u[\phantom{a} ])\lor(\forall y\exists v[\phantom{a}])$. 
If some operators $\forall x,\forall y, (\exists u/U),(\exists v/V)$ and $\circ$ (occurring in a tree) satisfy 0)-1)-2)-3) , we say that they form a \textbf{generalized Henkin pattern}. We will see many more examples in section \ref{GHTREES}, where we identify some classes of generalized Henkin prefixes that are NP-complete, and others that are in $\operatorname{FO}$. 


\begin{df}
A syntactical $IF$ tree is \textbf{coordinated} if it contains logical constants $(\forall  x/X),  \lor, (\forall y/Y), (\forall z/Z), (\exists u/U), (\exists w/W)$ such that:\\
0) $(\forall  x/X), (\forall y/Y), (\forall z/Z), (\exists u/U), (\exists w/W)$ do not all occur in a same branch\\
1) $\forall x\prec_T \circ \prec_T (\exists u/U), (\exists w/W)$\\
2) $u$ depends on $y$, but does not depend on $x, z, w$ \\
3) $w$ depends on $z$,  but does not depend on $x, y, u$\\
In case $(\exists u/U), (\exists w/W)$ occur in distinct disjuncts below $\lor$, we say the coordinated tree is of \textbf{first kind}; otherwise, we say it is of \textbf{second kind}.
\end{df}

\noindent Again, we say that five operators $(\forall  x/X),  \lor, (\forall y/Y), (\forall z/Z), (\exists u/U), (\exists w/W)$ occuring in a given tree form a \textbf{coordinated pattern} if they respect clauses 0)-1)-2)-3).  The main difference with respect to the case of generalized Henkin patterns is that there is one quantifier $(\forall x/X)$ which is not ``seen'' by the existential quantifiers. The other two universal quantifiers can occur either above or below the disjunction. In the case of generalized Henkin patterns, we forbid the possibility of having both universal quantifiers below the disjunction (or conjunction), so as to exclude trivially trivially first-order prefixes; here it is allowed for $\forall y$ and $\forall z$ to occur both below $\lor$ (as in the tree $\forall  x(\forall y (\exists u/x)\gap  \lor \forall z(\exists w/x)\gap)$ ), and this actually leads to second-order expressive power. More generally, in section \ref{CTREES} we will show that \emph{all} coordinated trees of first kind are NP-complete. The second  kind (e.g. the tree $\forall  x(\gap \lor(\forall y (\exists u/x)\gap  \land \forall z(\exists w/x)\gap)$  ) is less well understood; in section \ref{CTREES2} we show that some of these tree prefixes are first-order. 

\begin{df}
A regular tree is \textbf{modest} if it is neither  signalling, Henkin, generalized Henkin nor coordinated.
\end{df}

\begin{conv}
We apply the above terminology (Henkin, signalling, generalized Henkin, coordinated of first-second kind, modest) also to (regular) sentences, whenever the syntactical tree of the sentence has the corresponding property.

We will take sometimes the liberty, in the following, to call a sentence a tree, and to call a subformula a subtree (but not viceversa, at least in general).
\end{conv}

The importance of the taxonomy given in this section lies in the following result:

\begin{teo} \label{MODESTTREE}

a) All regular, modest $IF$ trees are in FO.

b) Every regular, modest $IF$ sentence can be transformed, by means of equivalence rules, into a first-order sentence. 
\end{teo}

\noindent The rather involved proof is postponed to Appendix \ref{APPMODEST}. An example is due to show that part b) of this theorem concretely improves over the criterion given in \cite{Sev2014} for recognizing first-orderness (i.e. checking whether a prenex $IF$ sentence is primary). A naive strategy for checking whether a (possibly non-prenex) regular $IF$ sentence is equivalent to some first-order sentence might be: transform the sentence into prenex normal form, and then check whether the resulting prenex sentence is primary (i.e. its syntactical tree is not Henkin nor signalling). However, this strategy gives some false negatives; in particular, it does not recognize some modest sentences as essentially first-order. Consider e.g. a sentence $\forall x(\forall y(\exists u/x)\psi \land \forall z(\exists v/x)\chi)$, with $\psi,\chi$ quantifier-free. By inspection one sees that this is a modest sentence. Furthermore, we can prove its equivalence to a first-order sentence as follows: first distribute $\forall x$ to obtain $\forall x\forall y(\exists u/x)\psi \land \forall x\forall z(\exists v/x)\chi$. Then by repeated applications of quantifier swapping (and by cancelling the slash sets of universal quantifiers) we obtain $\forall y\exists u\forall x\psi \land \forall z\exists v\forall x\chi$, which is first-order. These transformations apply to every sentence of this form, so what we have shown is that the tree $\forall x(\forall y(\exists u/x)\gap \land \forall z(\exists v/x)\gap)$ is in FO. However, the use of prenex form does not account for this fact. Our sentence can e.g. be transformed into  $\forall x\forall y(\exists u/x) \forall z(\exists v/x,y,u)(\psi \land\chi)$; here the prefix contains the Henkin pattern $\forall y,(\exists u/x),\forall z,(\exists v/x,y,u)$, and thus the sentence is not recognized as first-order using the criterion of \cite{Sev2014}. We also have to note that many different prenex forms can be obtained, and different prenex forms might yield different outcomes.\footnote{A systematical enumeration of prenex forms is further complicated by the fact that the \emph{strong} extraction rule presented here is not the only extraction rule that can be used in $IF$ logic. See \cite{ManSanSev2011}, Theorem 5.35 for a different, ``weak'' rule. The results in \cite{CaiDecJan2009} also show that a variety of intermediate versions of the extraction rule are available.}

\section{Extension lemma for tree prefixes}   \label{HIGHCOMP}

Because of theorem \ref{MODESTTREE}, we know that the search for tree prefixes with second-order expressive power can be limited to the Henkin, signalling, generalized Henkin and signalling classes. 
In the present section we move a first step in this direction: we generalize to tree prefixes the Extension Lemma of \cite{Sev2014} (which concerned quantifier prefixes); this is the main tool for the study of regular NP-complete prefixes. Roughly stated, it says that whatever can be expressed using a regular prefix, it can also be expressed using a larger regular prefix. In particular, if a regular tree can define NP-complete problems, also its extensions can. This result fails dramatically if irregular prefixes are allowed; for example, $\forall x\exists y(\exists z/x)$ is an NP-complete prefix, but its extension  $\forall x\exists y(\exists z/x)\forall x\forall y\forall z$ is first-order.
Also in the regular case, the proof is not straightforward: we must take carefully into account the signalling phenomena that may be introduced by the additional quantifiers.

First of all, we must make precise what we mean by an extension of a tree prefix.

\begin{df}
Let $T,U$ be tree prefixes. We say that $U$ \textbf{extends} $T$ if there is an injective function $\mu:T\rightarrow U$ such that: \\
1) for every quantifier node $(Qv/V)$ in $T$,  $\mu((Qv/V)) = (Qv/V')$ for some finite set of variables  $V'\subseteq V$; and for every connective node $c$, $\mu(c)$ is an occurrence of the same connective. \\
2) $\mu$ preserves the scope ordering $\prec_T$: if $c,d$ are two nodes in $T$, and $c\prec_T d$, then $\mu(c)\prec_U\mu(d)$.\\
3) if $(Qv/V)$ and $(Qw/W)$ occur in $T$, then the latter depends on the former if and only if $\mu((Qw/W))$ depends on $\mu((Qv/V))$.\footnote{The clause 2) makes our definition stricter than the corresponding notion for quantifier prefixes given by Sevenster. The correct generalization should allow some form of swapping of independent quantifiers. However, we will not need such subtleties.}   
\end{df}

In short, $U$ extends $T$ if it contains all the logical operators of $T$, possibly with swollen slash sets, and the operators that come from $T$ keep their original mutual relations of dependence and independence. 

\begin{example}
Let $T$ be $\forall x\exists y(\exists z/x)\gap$, and $U$ be $\forall w\forall x\exists y(\gap \land (\exists z/x)\gap)$. $U$ extends $T$ via the function $\mu$ that sends each operator from $T$ into an occurrence of the same operator in $U$, and the gap of $T$ into the right-hand gap of $U$. Notice that this extension preserves the property of being a signalling tree. However, $U$ has more intricate dependencies: for example, both $\exists y$ and $ (\exists z/x)$ now depend on $\forall w$.

Also $U':\forall w\forall x\exists y(\gap \land (\exists z/xw)\gap)$ is an extension of $T$, via a function $\mu'$ that differs from $\mu$ only in that it sends $(\exists z/x)$ to $(\exists z/xw)$. In this case, $z$ is independent of the new variable $w$; however the (signalling) dependencies between the operators that were already in $T$ have not changed. Instead, $U'':\forall w\forall x\exists y(\gap \land \exists z\gap)$ and $U''':\forall w\forall x\exists y(\gap \land (\exists z/x,y)\gap)$ are not extensions of $T$; the signalling pattern involving $x,y$ and $z$ is not anymore in place.
\end{example}

\begin{lem}[Extension Lemma for tree prefixes] \label{EXTENDLEMMA}
Let $T$, $U$ be regular $IF$ tree prefixes. Suppose $U$ extends $T$, and $\varphi$ is a completion of $T$. Then there are a completion $\varphi'$ of $U$ and, for every structure $M$ suitable for $\varphi$, an expansion $M'$ of $M$ such that $M\models\varphi$ iff $M'\models\varphi'$.  The signature of $M'$ only contains an additional constant symbol.
\end{lem}

The idea behind the proof is simple, but the details are rather involved; so the proof is postponed to Appendix \ref{APPEXTLEM}.

\begin{cor}
If a regular $IF$ tree prefix extends an \emph{NL}-hard (resp. \emph{P}, \emph{NP}-hard) positive initial tree, then it is \emph{NL}-hard (resp. \emph{P}, \emph{NP}-hard) itself. 
\end{cor}

\begin{proof}
Suppose the tree $T$ has a completion $\varphi$ which defines a C-hard problem (for C = NL, P or NP) over a class of structures K, and that tree $U$ extends $T$. Then by lemma \ref{EXTENDLEMMA} there is a completion $\varphi'$ of $U$ such that, for every structure $M$ suitable for $\varphi$, $M\models\varphi$ iff $M'\models\varphi'$, where $M'$ is $M$ expanded with a new constant. Since the increase in size of the structure, from $M$ to $M'$, is constant, the problem that $\varphi'$ defines on K$':=\{M'\ | \ M \in \operatorname K\}$  is C-hard.
\end{proof}

We can immediately apply this result to show that not only quantifier prefixes, but also tree prefixes are NP-complete if they are Henkin or signalling.

\begin{teo}  \label{HENKINCOMPLETE}
Any regular $IF$ tree prefix which is Henkin is \emph{NP}-complete.\footnote{We emphasize once more that, by saying that a tree prefix is NP-complete, we do not mean that it can express all NP problems. It just means that all problems it can define are in NP, and that \emph{at least one} NP-complete problem is defined by (a completion of) the tree prefix.}
\end{teo}

\begin{proof}
Any such tree $T$ extends a regular Henkin quantifier prefix. It was shown in \cite{Sev2014}, Theorem 22, that regular Henkin prefixes encode the NP-complete problem of 3-COLORABILITY.
So, from Lemma \ref{EXTENDLEMMA} it follows that $T$ is NP-hard. Since $IF$ sentences are in NP (Prop. \ref{FAGIN}), $T$ is NP-complete.
\end{proof}

\begin{teo} \label{SIGNALLINGCOMPLETE}
Any regular $IF$ tree prefix which is signalling is \emph{NP}-complete.
\end{teo}

\begin{proof}
Use Lemma \ref{EXTENDLEMMA} and Prop. \ref{FAGIN} again, on the basis that regular signalling quantifier prefixes codify the NP-complete problem EXACT COVER BY 3-SETS (\cite{Sev2014}, Theorem 23).\footnote{In \cite{BarHelRon2017} it is shown that also SAT and DOMINATING SET can be defined using the smallest signalling prefix.}
\end{proof}

The remaining sections are devoted to a systematical study of the classes of generalized Henkin and coordinated trees.

\section{Complexity of generalized Henkin trees} \label{GHTREES}

The minimal examples of generalized Henkin trees are of the following forms:

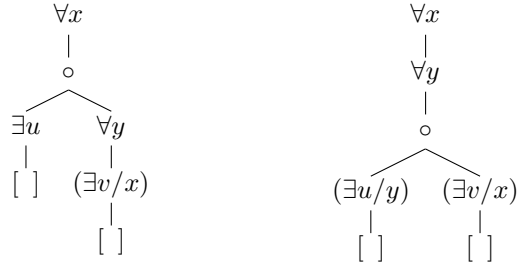
\vspace{5pt}

\Tree [.$\forall x$ [.$\circ$   [.$\exists u$ [.$[\phantom a]$ ] ] [.$\forall y$ [.$(\exists v/x)$ [.$[\phantom a]$   ]  ] ]  ] ]
\Tree [.$\forall x$ [.$\forall y$ [.$\circ$  [.$(\exists u/y)$ [.$[\phantom a]$   ]  ]   [.$(\exists v/x)$ [.$[\phantom a]$ ] ]  ] ] ]

\vspace{5pt}

\noindent where $\circ$ is either $\lor$ or $\land$. We will call GH1($\circ$) the first type, and GH2($\circ$) the second.

\subsection{Describing SAT by a minimal generalized Henkin sentence} \label{GH2OR}
Here we express the NP-complete problem SAT by means of an $IF$ sentence whose (positive initial) tree is generalized Henkin (specifically, GH2($\lor$)), but not Henkin nor signaling nor coordinated.\\
\\
SAT Problem: Given a proposition in conjunctive normal form,
 decide whether there is an assignment which satisfies the proposition.\\
\\
How we model the problem: each instance of it is a structure of signature $P^2,N^2,C^1,0,1$ ($0,1$ are constants denoting two distinct elements; $C(y)$: ``$y$ is a clause''; $\neg C(y)$: ``$y$ is a propositional letter''; $P(x,y)$ : ``$x$ occurs positively in $y$'';  $N(x,y)$ : ``$x$ occurs negatively in $y$''); we only allow structures such that for each clause $y$ there are at least two propositional letters $x$ such that $P(x,y)\lor N(x,y)$. It is well known that, even with this restriction, the SAT problem stays NP-complete.

For brevity, we shall write $O(x,y)$, ``$x$ occurs in $y$'', as a shortening for $\neg C(x) \land C(y) \land (P(x,y) \lor N(x,y))$.

The defining sentence is:
\[
\varphi: \forall x\forall y((\exists u/y)\psi_1 \lor (\exists v/x)\psi_2)
\]
where
\[
\psi_1: O(x,y) \land (P(x,y) \rightarrow u = 1) \land (N(x,y) \rightarrow u = 0)
\]
and
\[
\psi_2:  O(v,y) \land(O(x,y)\rightarrow x\neq v).
\]
\begin{teo}
If $M$ is a suitable structure, then $M \models \varphi$ iff $M$ is a ``yes'' instance of SAT.
\end{teo}

The idea behind this description is similar to that of Jarmo Kontinen's Theorem 4.3.3 from his PhD thesis (although, he deals with a very different kind of descriptive complexity; and although his method seems to capture just the 2-SAT problem). See \cite{Kon2010} or \cite{Kon2013} for a comparison. Think of $x$ as a propositional letter, $u$ as the truth value which is assigned to $x$, $y$ as a clause, $v$ as a propositional letter which corresponds to a literal of $y$ which is made true by the truth assignment. The left disjunct enforces $u$ to be a truth assignment; the $y$-uniformity of the function which picks $u$ guarantees that the assignment is correctly defined, i.e., a function of the propositional letters. The right disjunct ensures that, for every clause $\hat y$, there is at least one literal in it (corresponding to a prop.letter $\hat x$) which is made true by the assignment described by $u$; this is so because the formula $x\neq v$ enforces that at least one  pair of values $(\hat x, \hat y)$ for $(x,y)$  is sent to the left disjunct, which ensures that $\hat x$ is in the domain of the assignment. 

For the sake of the present proof, it will be convenient to adopt the following notation: if $X$ is a team and $v_1,\dots v_n$ a sequence of variables in the domain of $X$, we denote as $X(v_1,\dots,v_n)$ the relation $\{(s(v_1),\dots, s(v_n)) \ | \ s\in X\}$.

\begin{proof}
1) Suppose $M$ is a ``yes'' instance. Then there is a truth assignment $T$ on propositional letters which makes the proposition $\bigwedge \{c\in M| c\in C^M\}$ true. This means that to each clause $c$ we can associate a propositional letter $f(c)$ which either occurs positively in $c$ and $T(f(c)) = 1$, or it occurs negated in $c$ and $T(f(c)) = 0$. Let $R$ be $\{(f(c),c)|c\in M\}$, and $S = M^2 \setminus R$. Let $Y = \{s:\{x,y\}\rightarrow M \ | \ (s(x),s(y)) \in R \}$ and $Z = \{s:\{x,y\}\rightarrow M \ | \ (s(x),s(y)) \in S \}$ be the corresponding teams of domain $\{x,y\}$. They form a partition of $\{\emptyset\}[M/x,M/y]$. Let $Y':= Y[T/u]$; clearly $M,Y'\models \psi_1$. Let $g$ be any extension of $f$ to the whole $dom(M)$. Define $Z' = Z[g/v]$. Any triple $(\hat x,\hat y,\hat v)\in Z'(x,y,v)$ either is such that $\hat x$ does not occur in the clause $\hat y$, or, if it does, $\hat x$ is not $f(\hat y)$ (because the pair $(f(\hat y),\hat y)$ is not in $Z(x,y)$). So, $Z'$ satisfies $\psi_2$.

2) Suppose $M$ is a ``no'' instance. Let $Y,Z$ be any partition of $\{\emptyset\}[M/x,M/y]$; let $T$ be a $y$-uniform function $Y\rightarrow M$; let $g$ be an $x$-uniform function $Z\rightarrow M$. Define $Y', Z'$ from $X,Y,T,g$ as was done above. Since $T$ cannot be a satisfying assignment, there must be a clause $\hat y$ such that, for each propositional letter $x$, the triple $(x,\hat y,u)$ 
falsifies either $P(x,y) \rightarrow u = 1$ or $N(x,y) \rightarrow u = 0$ or $O(x,y)$. So, if $M,Y'\models \psi_1$, then for every $x\notin C^M$, $(x,\hat y,u)\notin Y'(x,y,u)$; so, $(x,\hat y)\notin Y(x,y)$; so, $(x,\hat y)\in Z(x,y)$. But then, if $(x,\hat y,v)\in Z'(x,y,v)$, by $x$-uniformity of $g$, we have that $(x',\hat y,v)\in Z'(x,y,v)$ for every propositional letter $x'$. So, $v$ must be equal to some such $x'$. Thus, $M, Z'\not\models O(x,y)\rightarrow x\neq v$: contradiction.
\end{proof}

\begin{cor}
The minimal generalized Henkin tree GH2($\lor$)

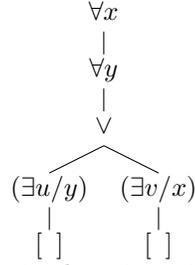
\vspace{5pt}

\Tree [.$\forall x$ [.$\forall y$ [.$\lor$  [.$(\exists u/y)$ [.$[\phantom a]$   ]  ]   [.$(\exists v/x)$ [.$[\phantom a]$ ] ]  ] ] ]

\vspace{5pt}

\noindent and any regular tree prefix extending it are NP-complete.
\end{cor}
It is perhaps of some interest that the SAT-describing sentence above can be rewritten as an H$_2^1$ Henkin prefix sentence
\[
\left(\begin{array}{cc}
\forall x & \exists u \\
\forall y & \exists v
\end{array}\right)
(\psi_1 \lor \psi_2).
\]

\noindent The paper \cite{KryVaa1989} introduced the so-called function quantifiers; in particular, the quantifier F$_2^1$, whose semantics is given by: $M\models$F$_2^1xyzw\psi(x,y,z,w)$ if and only if $\exists f\forall x \forall z\psi(x,f(x),z,f(z))$. It is unknown whether  F$_2^1$ is strictly less expressive than H$_2^1$.  Our SAT-defining sentence is an example of an H$_2^1$ sentence which cannot be reduced in any obvious way to a F$_2^1$ sentence, since the variables $u$ and $v$ describe here two very different functions. We are not aware of other examples of this kind in the literature.

\subsection{Disjunction-free Generalized Henkin trees} \label{GHAND}

\begin{conv}
From this section onwards, it will be useful to follow a convention: talking of a tree prefix $T$ which is a minimal representative of some class K of trees, we  will say that another tree $U$ is an extension$^*$ of $T$ if it is an extension of $T$ and furthermore it does not fall in any of the significant classes that are listed in the table at the end of the paper (signalling, Henkin, GH1($\land$), GH1($\lor$), GH2($\land$), GH2($\lor$), GH3, C1, C2, C1', modest) - except class K itself. So for example an extension$^*$ of the smallest signalling prefix $\forall x\exists y(\exists z/x)$ is any extension of it which is not Henkin, GH1($\land$),...nor  modest.   
\end{conv}

The minimal trees GH1($\circ$) and GH2($\land$) can be easily shown to be first-order (use quantifier distribution for GH1($\land$) and GH2($\land$); use the strong extraction rule in the longest branch of GH1($\lor$)). This tells us nothing about their extensions. We might conjecture that, if a tree falls in one of these classes but not in any other class that we have isolated (i.e., it is an extension$^*$ of GH1($\circ$) or GH2($\land$)), then it is in FO. We fall short of proving such results in full generality; for example, in the case of extensions of  GH1($\land$)  we prove this to hold only under the additional assumption that the tree in question does not contain disjunctions. In an earlier draft we claimed that the same suffices for requiring a GH2($\land$) to be in FO;  instead a stronger assumption is required. 
In order to express this condition, we follow \cite{JaaBar2019} and give a name to a particular extension of GH2($\land$), that we shall refer to as GH3:

\vspace{5pt}

\Tree [.$\forall x$ [.$\forall y$ [.$\exists i$ [.$\land$  [.$(\exists u/y,i)$ [.$[\phantom a]$   ]  ]   [.$(\exists v/x,i)$ [.$[\phantom a]$ ] ]  ] ] ] ]

\vspace{5pt}

\noindent In it, the extra quantifier $\exists i$ does not produce any Henkin or signalling patterns; and it blocks the possibility of distributing the universal quantifiers below the conjunction. In \cite{JaaBar2019} it is shown that the GH3 tree suffices to describe NP-complete problems (3-COLORING and SAT) and that a family of its extensions$^*$ captures ESO. Those results show that ``signalling by disjunction'' does not exhaust the sources of second-order expressive power of regular, non-prenex $IF$ logic. 



\begin{teo} \label{GHANDTEO}
1) If a regular tree prefix $T$ 
 does not contain disjunctions, and is not GH2($\land$), Henkin, nor signalling, then it has first-order complexity.                \\
2) If a regular tree prefix  $T$ 
 does not contain disjunctions, and is not in GH1($\land$), GH3, Henkin, nor signalling, then it has first-order complexity.               
\end{teo}

\begin{proof}
1) Suppose $T$ satisfies the hypotheses; then,  it is either modest or GH1($\land$). In the former case, it is in FO by theorem \ref{MODESTTREE}. In the latter, it contains at least one pattern 

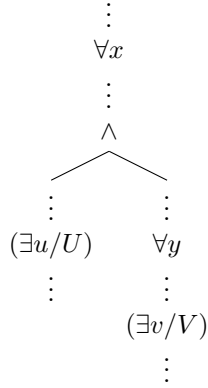
\vspace{5pt}

\Tree [.$\vdots$\\$\forall x$\\$\vdots$\\$\land$  [.$\vdots$\\$(\exists u/U)$\\$\vdots$ ] [.$\vdots$\\$\forall y$\\$\vdots$\\$(\exists v/V)$\\$\vdots$  ]  ]

\vspace{5pt}

\noindent where $x\notin U, y\notin V,x\in V$, witnessing that $T$ is GH1($\land$). Notice:\\
\\
1. There are, by assumption, no disjunctions between $\forall x$ and $(\exists v/V)$. \\
2. Every existential quantifier $(\exists w/W)$ between $\forall x$ and $\forall y$ is independent of $\forall x$ (otherwise either $\forall x,\exists w,\forall y, \exists v$ form a Henkin pattern, or $\forall x,\exists w, \exists v$ form a signalling pattern). Consequently, $\forall x$ can be pushed below any such quantifier by the quantifier swapping rule.\\
3.  $\forall x$ can be pushed below any other universal quantifier by quantifier swapping. \\
4. $\forall x$ can be pushed below any conjunction by means of quantifier distribution.\\
\\
None of these tranformations generates new dependence patterns
, nor disjunction symbols; so, applying them preserves the hypotheses of the theorem. Using these transformations, one can push $\forall x$ below 
$\land$, so that there is one less witness of the GH1($\land$) pattern.

Iterating the process, one can remove all witnesses of the GH1($\land$) pattern, until the resulting tree is modest (and thus of first-order complexity, Theorem \ref{MODESTTREE}).

2) Analogous. Point 2. by itself may be insufficient for the purpose of pushing $\forall x$ below the conjunction $\land$, since it does not exclude that there might be existential quantifiers depending on $\forall x$ and $\forall y$ between $\forall y$ and $\land$; it that case, it would be impossible to push $\forall x$ and/or $\forall y$ below $\land$. But this configuration does not arise, because of the assumption that $T$ is not GH3.
\end{proof}

This theorem adds to theorem \ref{MODESTTREE} in that allows recognizing some non-modest trees/sentences as FO; for example, $\forall z\forall x((\exists u/z)\gap \lor \forall y(\exists v/x)\gap)$ is a GH1($\land$) prefix which satisfies part 1) of the theorem, while $\forall z\forall x\forall y((\exists u/zy)\gap\land(\exists v/zx)\gap)$ is a GH2($\land$) prefix which satisfies part 2). As far as we could see, it seems impossible to tell that sentences with these prefixes are first-order just by taking prenex forms and checking that the tree is primary; the prenex transformations seem to always generate Henkin patterns, for the first tree, and signalling patterns for the second.

In the contrapositive, the theorem above tells us that the search for genuinely new second-order tree prefixes without disjunctions can be restricted to trees that are \emph{both} GH1($\land$) and GH2($\land$):

\begin{cor}
Let $T$ be a regular tree prefix without occurrences of disjunction and which is not signalling, Henkin nor GH3. If $T$ is not in $FO$, then $T$ is both GH1($\land$) and GH2($\land$). 
\end{cor}

\subsection{Conjunction-free GH1 trees} \label{GH1OR}

For GH1($\lor$), we conjecture that a dual result may hold: extensions$^*$ of GH1($\lor$) which do not contain \emph{conjunction} symbols are in FO.

\begin{conjecture} \label{GH1ORTEO} 
Suppose an IF regular tree prefix has no conjunctions, and it is not GH2($\lor$), coordinated, Henkin nor signalling. Then it is in the FO complexity class.
\end{conjecture}

\noindent However, we have no fully convincing proof of this statement.




\section{Coordinated trees of the first kind} \label{CTREES}

The minimal examples of coordinated trees of the first kind can have the following forms:

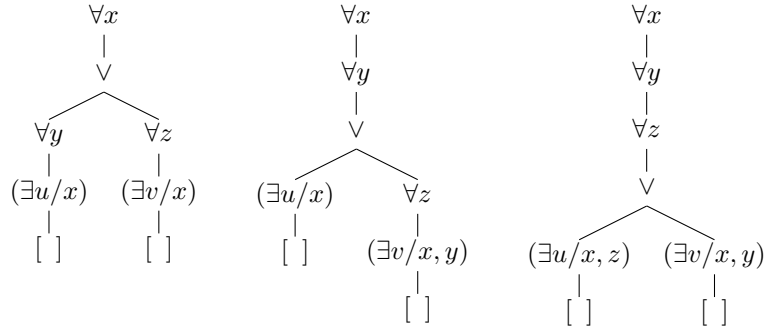
\vspace{5pt}

\Tree [.$\forall x$ [.$\lor$ [.$\forall y$ [.$(\exists u/x)$ [.$[\phantom a]$   ]  ] ]  [.$\forall z$ [.$(\exists v/x)$ [.$[\phantom a]$ ] ] ] ] ]
\Tree [.$\forall x$ [.$\forall y$ [.$\lor$  [.$(\exists u/x)$ [.$[\phantom a]$   ]  ]   [.$\forall z$ [.$(\exists v/x,y)$ [.$[\phantom a]$ ] ] ] ] ] ]
\Tree [.$\forall x$ [.$\forall y$ [.$\forall z$ [.$\lor$  [.$(\exists u/x,z)$ [.$[\phantom a]$   ]  ]    [.$(\exists v/x,y)$ [.$[\phantom a]$ ] ] ] ] ] ]

\vspace{5pt}

\noindent We call these trees (and the corresponding fragments of $IF$ logic) C1, C2, and C3, from left to right. It is apparent that C(C1)$\subseteq$ C(C2) $\subseteq$ C(C3). In the subsections we show that all extensions of these trees are NP-complete.

\subsection{SAT by coordinated trees}\label{CSAT}

First: observe that the coordinated tree C3 is an extension of the generalized Henkin tree GH2($\lor$).
So, by the Extension Lemma, it permits defining the SAT problem. So, all trees extending C3 are NP-complete (and we can exclude them from our classification, since they are a special case of GH2($\lor$) trees).\\
\\
Secondly: we give a different (but similar in spirit) description of SAT by means of the coordinated tree C2. This will prove that all trees extending C2 are NP-complete. We use the same notations and conventions as in the previous section, with the following exception\footnote{This difference is just related to our choice of names for the quantified variables; it has no deeper meaning.}: now $P(x,y)$ is interpreted as ``$y$ is a prop. letter occurring positively in the clause $x$'', and not viceversa; similarly for the relations $N$ and $O$. $O(x,y)$ is an abbreviation for $C(x)\land \neg C(y) \land (P(x,y)\lor N(x,y))$. We assume that each clause contains at least one literal. Then, the  SAT-defining sentence is:
\[
\theta: \forall x\forall y((\exists u/x)\chi_1 \lor\forall z(\exists v/xy) \chi_2)
\]
where
\[
\chi_1 : 
O(x,y)\land [ (P(x,y) \rightarrow u=1) \land (N(x,y) \rightarrow u=0)]
\]
and
\[
\chi_2 : (z=x \land O(x,y)) \rightarrow (v\neq y \land O(x,v)).
\]

\begin{teo}\label{TEOC2}
If $M$ is a suitable structure, then $M\models\theta$ iff $M$ encodes a ``yes'' instance of SAT. 
\end{teo}

\begin{proof}
1) Suppose $M$ encodes a ``yes'' instance of SAT. Then there is an assignment $T$ of truth values to the propositional variables which makes the conjunction of clauses true. This means that to each clause $a$ we can associate a proposition $b = g(a)$ (we functionally choose one) such that either $b$ occurs positively in $a$, and $T(b)=1$, or $b$ occurs negatively in $a$, and $T(b)=0$. We say that such pairs are in a relation $R(a,b)$. Now define $Y,Z\subseteq \{\emptyset\}[MM/xy]$ as: $Y:=\{\{(x,a),(y,b)\} \ | \ (a,b)\in R\}$, $Z:= \{\emptyset\}[MM/xy]\setminus Y$. Define the function $F:Y\rightarrow M$, by $F(s):= T(s(y))$ whenever $s(y)$ is a propositional letter, and arbitrarily otherwise. Then, by the comments on $T$ above, $M,Y[F/u]\models \chi_1$. 

 Define the function $G: Z[M/z]\rightarrow M$ as $G(s) := g(s(z))$. Then, if $s\in Z[MG/zv]$ and $s(z)=s(x)$, we have $s(v)=g(s(z))=g(s(x))$; therefore $s(y)\neq s(v)$, because $s(v)= g(s(x))$, and the pair $((x,s(x)),(v,g(s(x))))$ is in $Y$, so not in $Z$. Furthermore,
 by the definition of $G$, $(s(x),s(v)) = (s(x),g(s(x))\in R$, which implies $(s(x),s(v))\in O^M$. So $M, Z[MG/zv]\models \chi_2$.

2) Suppose $M\models\theta$. Then there are $Y,Z\subseteq \{\emptyset\}[MM/xy]$, a $x$-uniform function $F: Y\rightarrow M$ and a $xy$-uniform $G: Z[M/z]\rightarrow M$ such that $Y\cup Z=\{\emptyset\}[MM/xy]$, $M,Y[F/u]\models\chi_1$ and $M,Z[MG/zv]\models\chi_2$. Suppose for sake of contradiction that, for some clause $\hat a\in C^M$, $s\in Y$ implies $s(x)\neq \hat a$. By our assumption on structures, that in each clause at least one literal occurs, there must be an $s\in Z$ such that $b:=s(y)$ occurs in $\hat a$. Pick $s'\in Z[MG/zv]$ such that $s'(x) = \hat a$, $s'(y) = b$ and $s'(z) = s'(x)$. Since $M,Z[MG/zv]\models\chi_2$, we have $s'(v)\neq s(y)$ and $(s'(v),s(x))\in O^M$. By $xy$-uniformity of $G$, $s'(v)= G(s'_{\upharpoonright \{x,y,z\}})$ is different from all  $s''(y)$ such that $s''\in Z$. Then, the assignment $\{(x,\hat a),(y,s'(v))\}$ must be in $Y$, contradicting our hypothesis.

So, for each clause $a$, there is a propositional letter $g(a)$ occurring in $a$ such that $s_a:=\{(x,a),(y,g(a))\}\in Y$. Define $T(g(a)):= F(s_a)$, and extend it arbitrarily to a propositional assignment over propositional letters that are not of the form $g(a)$. Since $M,Y[MF/yu]\models \chi_1$ and $Y$ contains $s_a$ for each clause $a$, $T$ is an assignment that satisfies the instance of SAT which is encoded by $M$.
\end{proof}

In this proof we used the restriction that each clause contain at least one literal; eliminating such restriction on the class of structures would require the usage of an extra existential quantifier (independent of $x$) in each disjunct; the resulting tree would be an extension of the one considered, and (because of the right disjunct) either signalling or a Henkin tree. In any case, as before, also this variant of SAT is NP-complete. So:

\begin{cor}
The coordinated tree C2 (and any tree extending it) is NP-complete. 
\end{cor}

\subsection{NP-completeness of C1} \label{C1LHARD}

We do not know whether the SAT problem is definable by means of the coordinated tree C1; however, we show here that a different NP-complete problem, SET SPLITTING (see e.g. \cite{GarJoh1979}), is definable by means of C1. This result was obtained in collaboration with Lauri Hella, who kindly agreed on including it in this paper.\\
\\
Input: a set $A$, a family $\mathcal{B}\subseteq \wp(A)$ s.t., for every $B\in \mathcal{B}$, $card(B)\geq 2$.\\
\\
Measure of the input: $card(A\cup\mathcal{B})$.\\
\\
Problem: Is there a partition $\{U,V\}$ of $A$ such that, for each $B\in \mathcal{B}$, $B\cap U \neq\emptyset$ and  $B\cap V\neq\emptyset$?\\
\\
We encode input instances as structures of domain $A\cup\mathcal{B}$ (with $\mathcal{B}\subseteq \wp(A)$ such that each of its element has at least cardinality 2) which interprets in the obvious way unary predicates $A$ and $\mathcal{B}$, and  a binary ``set membership'' relation $R_\in$ (with the restriction that, if  $(a,B)\in R_\in$, then $a\in A$, $B\in \mathcal{B}$ and $a\in B$).

The requirement that the sets in $\mathcal{B}$ have at least cardinality 2 is our addition to the original problem; it obviously does not decrease its complexity, and it makes the problem easier to define in our fragment of $IF$ logic.

The defining sentence is:
\[
\eta: \forall x(\forall y(\exists u/\{x\})\epsilon_1 \lor \forall z(\exists v/\{x\})\epsilon_2)
\]
where
\[
\epsilon_1: (A(x) \land \mathcal{B}(y)) \rightarrow (u\neq x \land R_\in(u,y))
\] 
and
\[
\epsilon_2: (A(x) \land \mathcal{B}(z)) \rightarrow (v\neq x \land R_\in(v,z))
\] 

\begin{teo} \label{C1ISNP}
For every suitable structure $M$, $M\models\eta$ iff $M$ encodes a ``yes'' instance of SET SPLITTING.
\end{teo}

\begin{proof}
$\Leftarrow$) Let $M$ be a ``yes'' instance of SET SPLITTING. Let $\{U,V\}$ be a partition of $A$ which satisfies the requirement of the problem: for every $B\in \mathcal{B}$, $B\cap U \neq\emptyset$ and  $B\cap V\neq\emptyset$. For each $B\in \mathcal{B}$, choose a $u_B\in B\cap V$ and a $v_B\in B\cap U$. Define teams $Y:=\{s\in\{\emptyset\}[M/x] \ | \ s(x)\in U\}$ and $Z:=\{\emptyset\}[M/x] \setminus Y$; they form a partition of $\{\emptyset\}[M/x]$. Let $F:Y[M/y]\rightarrow M$ be defined as $F(s) := u_B$ if $s(y) = B\in\mathcal{B}$, and as an arbitrary function of $y$ otherwise. Let $G:Z[M/z]\rightarrow M$ be defined as $G(s) := v_B$ if $s(z) = B\in\mathcal{B}$, and arbitrarily otherwise. Since the $u_B$s are in $V$, they are not in $U$, and so every $s\in Y[MF/yu]$ is such that, if $s(y)\in\mathcal{B}$, then $s(u) = F(s_{\upharpoonright\{x,y\}}) = u_B \neq s(x)\in U$. So,  $M,Y[MF/yu]\models \epsilon_1$. A symmetrical argument shows that $M,Z[MG/zv]\models \epsilon_2$.

$\Rightarrow$) Suppose $M\models \eta$. Then there are $Y,Z\subseteq \{\emptyset\}[M/x]$ such that $Y\cup Z = \{\emptyset\}[M/x]$, and $x$-uniform functions $F:Y[M/y]\rightarrow M$ and $G: Z[M/z]\rightarrow M$, such that $M,Y[MF/yu]\models \epsilon_1$ and $M,Z[MG/zv]\models \epsilon_2$.
Define $U$ as $\{a\in A \ | \ \exists s\in Y(s(x)=a)\}$, and $V:= A\setminus U$. Since $U\cup V=A$, at least one out of $U$ and $V$ is nonempty. We suppose w.l.o.g. that $U$ is nonempty, which implies that $Y$ is nonempty. 

Let $B\in\mathcal{B}$. Let $s_B\in Y[MF/yu]$ be an assignment such that $s_B(y) = B$ and $s_B(x)\in A$ (there is at least one such $s_B$, because of the nonemptyness of $Y$ and the fact that $y$ is universally quantified). The fact that $M,s_B\models \epsilon_1$ implies that $s_B(u)\in s_B(y)=B$ and $s_B(u)\neq s_B(x)$; since $s_B(u) = F(s_{\upharpoonright \{x,y\}})$, the $x$-uniformity of $F$ implies that $s_B(u)\neq s(x)$ for each $s\in Y[MF/yu]$, that is,  $s_B(u)\neq a$ for all $a\in U$. So $s_B(u)\in B\cap V$. 

This furthermore implies that $V\neq\emptyset$. So, by a symmetric argument one can prove the existence of one element in $B\cap U$. 
\end{proof}

\begin{cor}
The coordinated tree C1, as all trees extending it, is NP-complete.
\end{cor}

This result, together with the Extension Lemma (\ref{EXTENDLEMMA}), yields an alternative,  more indirect proof of theorem \ref{TEOC2}.
  
This concludes the classification of coordinated trees of first kind up to reduction closure. However, since plausibly the minimal coordinated trees do not capture all NP problems, it might be of interest that we found a description of an L-complete problem, 2-COLORABILITY, by means of the minimal C1 tree. This problem is known not to be in FO. The 2-COLORABILITY problem can be described as follows: given a graph $G = (V,E)$, show that $V$  can be decomposed as a partition into subsets $A,B$ such that $A^2\cap E =\emptyset$ and $B^2\cap E =\emptyset$ (i.e., there are no edges between vertices of $A$, and similarly for $B$).

The defining sentence, in the language of graphs, is:
\[
\xi: \forall x(\forall y(\exists u/\{x\})\xi_1\lor\forall z(\exists v/\{x,y\})\xi_2)
\]
where 
\[
\xi_1: E(x,y) \rightarrow (u=y \land u\neq x)
\]
and
\[
\xi_2: E(x,z) \rightarrow (v=z \land v\neq x).
\]

\begin{teo}
A graph structure $M$ satisfies $\xi$ if and only if it encodes a ``yes'' instance of 2-COLORABILITY.
\end{teo}

\begin{proof}
1) Suppose $M$ is a ``yes'' instance of 2-COLORABILITY. Then, its domain can be partitioned into two subsets $A,B$ such that $c\in A$ plus $(c,d)\in E^M$ implies $d\in B$, and viceversa.

Define $Y:= \{s\in \{\emptyset\}[M/x] \ | \ s(x)\in A\}$ and $Z=\{\emptyset\}[M/x] \setminus Y$.
Define $F:Y[M/y]\rightarrow M$, $F(s) := s(y)$, and $G:Z[M/z]\rightarrow M$, $G(s) := s(z)$.

Notice that,  if $s\in Y$, then $s(x)\in A$; and if $(s(x),s(y))\in E^M$, then $s(y)\in B$; so, since $A$ and $B$ are disjoint, $s(y)\neq s(x)$. Furthermore, $s(u)=s(y)$ by the definition of $F$. Thus, $M,Y[MF/yu]\models \xi_1$. The proof that $M,Z[MG/zv]\models \xi_2$ is completely analogous. 

2) Suppose $M\models \xi$.
Then there are $Y,Z\subseteq \{\emptyset\}[M/x]$, a $x$-uniform function $F: Y[M/y]\rightarrow M$ and a $x$-uniform $G: Z[M/z]\rightarrow M$ such that $Y\cup Z=\{\emptyset\}[M/x]$, $M,Y[MF/yu]\models\xi_1$ and $M,Z[MG/zv]\models\xi_2$. By downward closure, we can assume that $Y\cap Z=\emptyset$. Let $A:=\{a\in M \ | \ \{(x,a)\}\in Y\}$, and $B=M\setminus A$. Now suppose, for the sake of contradiction, that $a\in A$, $(a,c)\in E^M$ and $c\in A$. There is an $s\in Y[MF/yu]$ such that $s(x)=a$ and $s(y)=c$. Since $M,Y[MF/yu]\models\xi_1$ and $M,s\models E(x,y)$, we have $s(u)=s(y)$ and $s(u)\neq s(x)$. Since $s(u) = F(s_{\upharpoonright \{x,y\}})$ and $F$ is $x$-uniform, we have $s(u) \neq s'(x)$ for all $s'\in Y[MF/yu]$. But $s(u)=s(y)$; so, $s(y)$ is different from all $s''(x)$ such that $s''\in Y$. Thus $c=s(y) \in B$: a contradiction.

Similarly one proves that $b\in B$, $(b,c)\in E^M$ implies $c\in A$.
\end{proof}

\section{Coordinated trees, second kind} \label{CTREES2}

A coordinated tree is of the second kind if it contains some logical operators $\forall x,\forall y,\forall z,\lor,(\exists u/U),(\exists v/V)$ that form a coordinated pattern, and such that $(\exists v/V),(\exists v/V)$ occur in the same disjunct below $\lor$. Since the definition of coordinated pattern excludes the trivial (Henkin) case that $(\exists v/V),(\exists v/V)$ are in a same branch of the tree, we must suppose that below $\lor$ there is a connective $\circ$ such that $(\exists u/U)$ occurs (say) in the left subformula below $\circ$, and $(\exists v/V)$ occurs in the right subformula. If $\circ$ is a disjunction, then the tree is also first kind; so it is NP-complete by the results of the previous section. We focus then on the case $\circ = \land$.  Taking into account all the different positions in which $\forall y,\forall z$ can occur, and ignoring permutations of quantifiers of the same kind, we can isolate six minimal coordinated trees of the second kind:

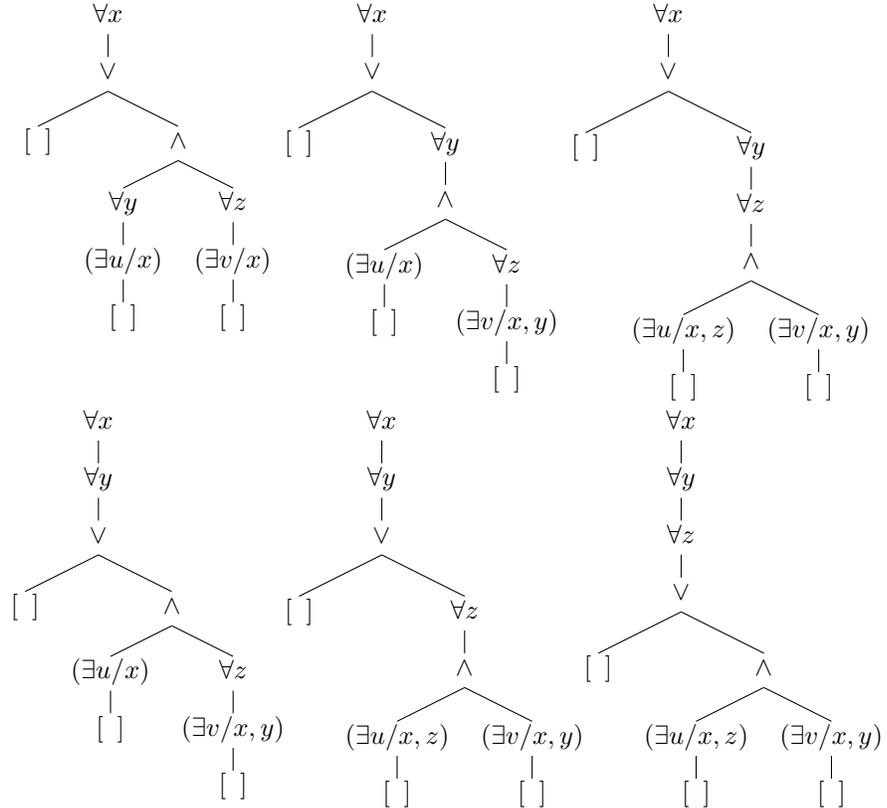
\vspace{5pt}

\Tree [.$\forall x$ [.$\lor$ [.$[\phantom{a}]$ ]  [.$\land$ [.$\forall y$ [.$(\exists u/x)$ [.$[\phantom a]$   ]  ] ]  [.$\forall z$ [.$(\exists v/x)$ [.$[\phantom a]$ ] ] ] ] ] ]
\Tree [.$\forall x$ [.$\lor$ [.$[\phantom{a}]$ ]  [.$\forall y$ [.$\land$  [.$(\exists u/x)$ [.$[\phantom a]$     ] ]   [.$\forall z$ [.$(\exists v/x,y)$ [.$[\phantom a]$ ] ] ] ] ] ] ] 
\Tree [.$\forall x$ [.$\lor$ [.$[\phantom{a}]$ ]  [.$\forall y$ [.$\forall z$ [.$\land$  [.$(\exists u/x,z)$ [.$[\phantom a]$     ] ]   [.$(\exists v/x,y)$ [.$[\phantom a]$ ] ]  ] ] ] ] ] 

\vspace{5pt}

\Tree [.$\forall x$ [.$\forall y$ [.$\lor$ [.$[\phantom{a}]$ ] [.$\land$  [.$(\exists u/x)$ [.$[\phantom a]$     ] ]   [.$\forall z$ [.$(\exists v/x,y)$ [.$[\phantom a]$ ] ] ] ]  ] ] ]
\Tree [.$\forall x$ [.$\forall y$ [.$\lor$ [.$[\phantom{a}]$ ] [.$\forall z$ [.$\land$  [.$(\exists u/x,z)$ [.$[\phantom a]$     ] ]    [.$(\exists v/x,y)$ [.$[\phantom a]$ ] ] ] ]  ]  ] ]
\Tree [.$\forall x$ [.$\forall y$ [.$\forall z$ [.$\lor$ [.$[\phantom{a}]$ ]  [.$\land$  [.$(\exists u/x,z)$ [.$[\phantom a]$     ] ]    [.$(\exists v/x,y)$ [.$[\phantom a]$ ] ] ] ]  ]  ] ]

\vspace{5pt}

\noindent From left to right, we call these trees C1', C2', C3', C4', C5', C6'. 
Before the reader starts worrying because of this explosion of cases, we point out that only the C1' case is genuinely new, while C2', C3', C4', C5', C6' are extensions of either tree GH1($\land$) or GH2($\land$); they fall into cases that we had already left open before. Notice, furthermore, that C(C1') $\subseteq$ C(C2') $\subseteq$ C(C3') $\subseteq$ C(C5') $\subseteq$ C(C6') and C(C1') $\subseteq$ C(C2') $\subseteq$ C(C4') $\subseteq$ C(C5') $\subseteq$ C(C6').

We prove that the tree C6' (and thus C5', C4', C3', C2', C1') is in FO (the key idea of the proof is due to Lauri Hella). The result tells us nothing about extensions of these trees.

\begin{teo} \label{C6TEO}
The trees C1', C2', C3', C4', C5', C6' are in FO.
\end{teo}

\begin{proof}
We prove that sentences  which begin with tree C6', that is, are of the form 
\[
\varphi: \forall x\forall y\forall z(\psi_1(x,y,z) \lor ((\exists u/\{x,z\})\psi_2(x,y,z,u) \land (\exists v/\{x,y\})\psi_3(x,y,z,v))),
\]
 with $\psi_1,\psi_2,\psi_3$ quantifier-free, are equivalent to sentences of the form 
\[
\varphi': \forall y\forall z( ((\exists u/\{z\})\forall x(\psi_1(x,y,z) \lor\psi_2(x,y,z,u)) \land
\]
\[ 
(\exists v/\{y\})\forall x(\psi_1(x,y,z) \lor \psi_3(x,y,z,v)))).
\]

\noindent Notice now that $\varphi'$ can also be obtained as completion of the positive initial tree

\Tree [.$\forall x$ [.$\forall y$ [.$\land$  [.$(\exists u/x,z)$ [.$\forall x$ [.$[\phantom{a}]$ ] ] ]   [.$(\exists v/x,y)$ [.$\forall x$ [.$[\phantom{a}]$  ] ] ]  ]  ] ]

\vspace{5pt}

\noindent which is a disjunction-free extension$^*$ of the GH2($\land$) tree, and thus in FO, by Theorem \ref{GHANDTEO}. So, $\varphi$ itself is equivalent to a first-order sentence.

We have to prove the above equivalence.

$\Longrightarrow$) $M\models \varphi$ iff there are teams $X_1,X_2\subseteq \{\emptyset\}[MMM/xyz]$ such that $X_1 \cup X_2 = \{\emptyset\}[MMM/xyz]$, a $\{x,z\}$-uniform function $F:X_2\rightarrow M$, and a $\{x,y\}$-uniform function $G:X_2\rightarrow M$, such that $M,X_1\models \psi_1(x,y,z)$, $M,X_2[F/u]\models \psi_2(x,y,z,u)$, and $M,X_2[G/v]\models \psi_3(x,y,z,v)$. 

Now fix an $a\in M$, and define functions $F',G':\{\emptyset\}[MM/yz]\rightarrow M$ as $F'(s) := F(s(a/x))$ and $G'(s) = G(s(a/x))$. Obviously $F'$ is $z$-uniform and $G'$ is $y$-uniform.  

Define teams $X_2':= \{s\in \{\emptyset\}[MMF'M/yzux] \ | \  M,s \models \psi_2(x,y,z)\}$ and $X_1' :=\{\emptyset\}[MMF'M/yzux] \setminus X_2'$. We have to verify that, then, $M,X_1'\models \psi_1(x,y,z,u)$. Suppose this is not the case, that is, there is an assignment $s\in X_1'$ such that $M,s\not\models\psi_1(x,y,z,u)$. By locality of quantifier-free formulas, $M,s_{-u}\not\models \psi_1(x,y,z,u)$ (where $s_{-u}$ is the assignment $s$ restricted to $dom(s)\setminus\{u\}$). This implies that $s_{-u}\in X_2$; so, that $s\in X_2'$; this contradicts the initial assumption that $s\in X_1'$. 

One can then analogously define $X_4':= \{s\in \{\emptyset\}[MMG'M/yzvx] | M,s \models \psi_3(x,y,z)\}$ and $X_3' :=\{\emptyset\}[MMG'M/yzvx] \setminus X_4'$, and prove that $M,X_3'\models \psi_1(x,y,z,v)$.

$\Longleftarrow$) $M\models\varphi'$ iff there are functions $F',G':\{\emptyset\}[MM/yz]\rightarrow M$ ($F'$ $z$-uniform, and $G'$ $y$-uniform) such that $M, \{\emptyset\}[MMF'M/yzux]\models \psi_1(x,y,z)\lor\psi_2(x,y,z,u)$, and $M, \{\emptyset\}[MMG'M/yzvx]\models \psi_1(x,y,z)\lor\psi_3(x,y,z,v)$. Calling $X_1':= \{s\in \{\emptyset\}[MMF'M/yzux] | M,s\models \psi_1(x,y,z)\}$ and $\overline X_1':= \{s\in \{\emptyset\}[MMG'M/yzvx] | M,s\models \psi_1(x,y,z)\}$, the last two statements above are equivalent to the existence of a team $X_2'\subseteq  \{\emptyset\}[MMF'M/yzux]$ such that $X_1\cup X_2' =  \{\emptyset\}[MMF'M/yzux]$ and $M,X_2'\models \psi_2$, and, respectively,  to the existence of a team $X_3'\subseteq  \{\emptyset\}[MMG'M/yzvx]$ such that $\overline X_1'\cup X_3' =  \{\emptyset\}[MMG'M/yzvx]$ and $M,X_3'\models \psi_3$.

Let $X_1 := \{s\in\{\emptyset\}[MMM/xyz] | M,s\models\psi_1(x,y,z)\}$. Let $X_2$ be its complement $\{\emptyset\}[MMM/xyz] \setminus X_1$.
 Define $F:X_2\rightarrow M$ as $F(s) := F'(s_{-x})$ and $G:X_2\rightarrow M$ as $G(s) := G'(s_{-x})$. Obviously $F'$ is $\{x,z\}$-uniform and $G'$ is $\{x,y\}$-uniform. 

Does $M, X_2[F/u]\models \psi_2(x,y,z,u)$? Yes, because $s\in X_2[F/u]$ implies $s\in \{\emptyset\}[MMF'M/yzux]$, and we already know that $M,\{\emptyset\}[MMF'M/yzux]\models \psi_2(x,y,z,u)$. Similarly, one can see that $M, X_2[G/u]\models \psi_3(x,y,z,v)$.

\end{proof}

   \section{Conclusions}

In this paper we have classified, up to reduction closure, many of the syntactical fragments of $IF$ logic that are individuated by positive initial trees (see the table at the end). All the tree prefixes that we have examined fall in the FO/NPC dichotomy. So, the question whether positive initial trees respect the dichotomy is still open.

One of the main contributions of the paper is the individuation of new patterns which allow defining second-order properties in $IF$ logic: we have found three patterns (GH2($\lor$), C2, C1) which express NP-complete problems even though they contain no Henkin nor signalling quantifier patterns. They do this by using some forms of ``signalling by disjunction''. (As already pointed out before, a further pattern GH3, which uses conjunctions instead of disjunctions, has been recently discovered in \cite{JaaBar2019}). For all we know, there might still be other unrecognized higher-order patterns (to be found among extensions of the GH1($\land$), GH1($\lor$), GH2($\land$) and C1' trees). We also point out that the descriptions of NP-complete problems we have found are quite atypical; in particular, they can be easily translated into H$_2^1$ sentences (H$_2^1$ being the smallest, four-place Henkin quantifier) but not so easily into F$_2^1$ sentences (F$_2^1$ being the smallest function quantifier, see \cite{KryVaa1989}).

For what regards trees of low complexity, our theorem on modest trees (\ref{MODESTTREE}) together with further results on generalized Henkin trees (\ref{GHANDTEO}) and coordinated trees of the second kind (
\ref{C6TEO}), provides a rather general sufficient (and effective) criterion for recognizing $IF$ sentences that have first-order expressive power, thus extending some criteria that come from earlier literature: the primality test of Sevenster (\cite{Sev2014}), on one side, and the Knowledge Memory test (\cite{Bar2013}) (which in turn extended the earlier Perfect Recall test, \cite{ManSanSev2011}). A different criterion for first-orderness is given by checking the absence of \emph{broken signalling sequences}, in the sense of \cite{Bar2013}; a moment of thought shows that this is also a special case of the modest tree criterion. The search for increasingly general sufficient, effective criteria for first-orderness has an interest because recognizing the $IF$ sentences (resp. ESO sentences, etc.) that are equivalent to first-order ones is an \emph{undecidable} problem.\footnote{We do not know where to find an easy proof of this fact in the literature (but see \cite{Cha1991}).}

The present results can be seen as a step forward in the understanding of fragments of $IF$ logic. Future work should be addressed to a more systematical understanding of the classes GH1, GH2($\land$) and C1', although it is not clear at present whether a complete and reasonable classification of the regular tree prefixes is possible. 

In the somewhat long time that has elapsed since the archiving of an earlier draft of this paper, some significant further progress has been made in the classification of fragments of $IF$ logic induced by quantifier and tree prefixes. We have already mentioned the discovery of the NP-complete GH3 tree prefix (\cite{JaaBar2019}), which, for completeness, we include in the summary table at the end of the paper.  \cite{BarHelRon2018} developed tools for the study of \emph{ir}regular prefixes, isolated some new minimal NP-complete prefixes that do not occur in the regular case (e.g. ``long signalling sequences'') and a putative counterexample to the FO/NP-complete dichotomy. Aside from the issue of tractability, a complementary problem has been studied: whether classes of extensions$^*$ of the significant patterns can be used to capture the whole existential second-order logic. These kinds of results typically are obtained by using the $IF$ patterns to explicitly simulate a complete set of Henkin quantifiers: this has been shown to be possible using extensions$^*$ of the signalling (\cite{BarHelRon2017}), GH2($\lor$) (\cite{BarHelRon2018}), C1 and GH3 (\cite{JaaBar2019}) prefixes.

Further work might be directed at 
finding exact characterizations of the expressive power of fragments (not just up to reduction closure). Secondly, it might be interesting to investigate what happens abandoning the restriction that trees be positive initial; although, surely in this case a satisfactory classification is impossible (a complete classification of syntactical trees would yield in particular a sufficient and necessary criterion for the first-orderness of sentences -- which, as we said, is an undecidable problem). One interesting example, in this sense, is the tree
\[
\forall x\exists \alpha\forall z(\exists\beta/\{x,\alpha\})((\alpha = 0 \lor \alpha = 1) \land (\beta = 0 \lor \beta =1) \land [\phantom a])
\]
which is equivalent to the smallest of the so-called \emph{narrow Henkin quantifiers}. From the  results of \cite{BlaGur1986}, it follows that this tree is NL-complete, a possibility that, so far, we have not individuated among regular tree prefixes. 

A third direction of work might be the analysis of quantifier and tree prefixes of logics similar to $IF$, such as the system $IF^*$ (see e.g. \cite{CaiDecJan2009}), which also allows slashed connectives, or Dependence-friendly logic (\cite{Vaa2007}), or their extensions via generalized quantifiers (see \cite{Eng2012}, \cite{Sev2014}). We also hope that the understanding of fragments of $IF$ logic might be of help for the analysis of other logics that cover NP, and that are structurally very different. One example is given by the logics of imperfect information based on atoms, such as Dependence logic (\cite{Vaa2007}), Independence logic (\cite{GraVaa2013}) or Inclusion logic (\cite{Gal2013}), whose higher-order expressive power is in part 
 generated at the level of quantifier-free formulas. The other main example is existential second-order logic; in particular, its functional version, for which a prefix approach only yields a not too interesting dichotomy.


In the following pages, a table summarizes all that we know about regular,
positive initial tree prefixes. Remember that, as a convention, if T is the name
of a specific tree, we refer to its \emph{extensions$^*$} to mean trees that extend T and do
not fall in any other of the categories described in the table. 

\section*{Acknowledgements}
Many thanks are due to Lauri Hella, for at least three reasons: 1) the key ideas of theorems \ref{C1ISNP} and \ref{C6TEO}, 2) the suggestion of re-reading Jarmo Kontinen's complexity results, and 3) many discussions on basic but tricky aspects of computational and descriptive complexity.

The author also wishes to thank Erich Gr\"adel for discussions about the ESO(KROM) fragment of second-order logic, and Reijo Jaakkola for many discussions on the fragments of $IF$ logic. These conversations prevented some mistakes from appearing in this article. Finally, we are grateful to Gianluca Grilletti for reading carefully part of the manuscript and for suggesting many improvements to the presentation.

\small
\begin{tabular}{|c|c|}
\hline
 \textbf{Tree} & \textbf{Complexity}  \\
\hline
\hline
\multicolumn{2}{|l|}{\textbf{Henkin}}\\
\hline
\pbox{20cm}{$\phantom{a}$\\$\forall x\exists y\forall z(\exists w/x,y)$, \\ $\forall x\forall z(\exists y/z)(\exists w/x,y)$, \\ $\forall x\forall z(\exists w/x)(\exists y/z,w)$ \\ and their extensions\\}  & NP-complete (3-COLORING)  \\
\hline
\multicolumn{2}{|l|}{\textbf{Signaling}}\\
\hline
\pbox{20cm}{$\phantom{a}$\\$\forall x\exists y(\exists z/x)$ \\ and its extensions\\} & \pbox{20cm}{NP-complete \\ (EXACT COVER BY 3-SETS, \\ SAT, DOMINATING SET)}  \\
\hline
\multicolumn{2}{|l|}{\textbf{Generalized Henkin}}\\
\hline
\pbox{20cm}{$\phantom{a}$\\ GH1($\land$): \Tree [.$\forall x$  [.$\land$  [.$\forall y$ [.$(\exists v/x)$ [.$[\phantom a]$   ] ]  ]   [.$\exists u$ [.$[\phantom a]$ ] ]  ] ] 
\\  and its disjunction-free\\ extensions$^*$\\}   & FO \\
\hline
\pbox{20cm}{$\phantom{a}$\\ GH2($\land$): \Tree [.$\forall x$ [.$\forall y$ [.$\land$  [.$(\exists v/x)$ [.$[\phantom a]$   ]  ]   [.$(\exists u/y)$ [.$[\phantom a]$ ] ]  ] ] ]
\\  and its disjunction-free\\ extensions$^*$\\}   & FO \\
\hline
\pbox{20cm}{$\phantom{a}$\\ GH3: \Tree [.$\forall x$ [.$\forall y$ [.$\exists i$ [.$\land$  [.$(\exists u/y,i)$ [.$[\phantom a]$   ]  ]   [.$(\exists v/x,i)$ [.$[\phantom a]$ ] ]  ] ] ] ] \\ and its extensions \\} & NP-complete (3-COLORING,  SAT) (\cite{JaaBar2019})   \\
\hline
\end{tabular}

\begin{tabular}{|c|c|}
\hline
\pbox{20cm}{$\phantom{a}$\\ GH1($\lor$): \Tree [.$\forall x$  [.$\lor$  [.$\forall y$ [.$(\exists v/x)$ [.$[\phantom a]$  ] ]  ]   [.$\exists u$ [.$[\phantom a]$ ] ]  ] ] 
\\ \\} &  FO \\
\hline
\pbox{20cm}{$\phantom{a}$\\ GH2($\lor$):  \Tree [.$\forall x$\\$\forall y$ [.$\lor$  [.$(\exists v/x)$ [.$[\phantom a]$   ]  ]   [.$(\exists u/y)$ [.$[\phantom a]$ ] ]  ] ] 
\\ and its extensions\\} & NP-complete (SAT) \\
\hline
 \pbox{20cm}{  \phantom{a}  Extensions$^*$ of GH1($\land$) or GH2($\land$) \\ without disjunctions \\} & FO  \\
\hline
 \pbox{20cm}{  \phantom{a}  Extensions$^*$ of GH1($\land$) or GH2($\land$) \\ with disjunctions \\ } & ???  \\
\hline
 \pbox{20cm}{ \phantom{a}  Extensions$^*$ of GH1($\lor$) \\  } & ??? \\
\hline
\multicolumn{2}{|l|}{\textbf{Coordinated}}\\
\hline
\pbox{20cm}{ \vspace{5pt} C1: \Tree [.$\forall x$ [.$\lor$ [.$\forall y$ [.$(\exists u/x)$ [.$[\phantom a]$   ]  ] ]  [.$\forall z$ [.$(\exists v/x)$ [.$[\phantom a]$ ] ] ] ] ] \\ and its extensions \\} & \pbox{20cm}{NP-complete (SET SPLITTING) }  \\
\hline
\end{tabular}

\begin{tabular}{|c|c|}
\hline
\pbox{20cm}{\vspace{5pt} C2: \Tree [.$\forall x$\\$\forall y$ [.$\lor$  [.$(\exists u/x)$ [.$[\phantom a]$   ]  ]   [.$\forall z$ [.$(\exists v/x,y)$ [.$[\phantom a]$ ] ] ] ] ] 
\\ and its extensions \\} & \pbox{20cm}{NP-complete (SAT, SET SPLITTING)}  \\
\hline
\pbox{20cm}{\vspace{5pt} C1': \Tree [.$\forall x$ [.$\lor$ [.$[\phantom{a}]$ ]  [.$\land$ [.$\forall y$ [.$(\exists u/x)$ [.$[\phantom a]$   ]  ] ]  [.$\forall z$ [.$(\exists v/x)$ [.$[\phantom a]$ ] ] ] ] ] ] 
 \\ 
\vspace{5pt}
} &  \phantom{aaaaaaaaaaaaa}FO \phantom{aaaaaaaaaaaa} \\
\hline
\pbox{20cm}{\phantom{a}  Extensions$^*$ of C1' \\ }  & C2'-C6' are FO; in general: ???  \\
\hline
\pbox{20cm}{\phantom{a}  \textbf{Modest}   \phantom{a} \\ }  & FO \\
\hline
\end{tabular}


\normalsize




\bibliographystyle{elsart-num-sort}

\bibliography{bibliotrees}


\pagebreak

\Large
\begin{center}
\textbf{APPENDIX}
\end{center}
\normalsize

\appendix

\renewcommand{\thesection}{\Alph{section}}

\section{Proof of theorem \ref{MODESTTREE}} \label{APPMODEST}

The key to proving theorem \ref{MODESTTREE} will be to show that it is possible to transform a modest sentence into prenex form while preserving the property of being modest. The following lemma shows that many of the tranformations that we consider in this paper (with the notable exception of quantifier extraction)  do not introduce new Henkin, signalling, generalized Henkin or coordinated patterns (in short: \textbf{non-modest} patterns). Therefore, when applied to a  modest sentence, they produce a new modest sentence. This lemma will be mostly applied implicitly in what follows.

\begin{lemma} \label{SIMPLERULESPRESERVEMODESTY}
The following tranformations do not introduce new non-modest patterns in regular $IF$ sentences:
\begin{enumerate}[a)]

\item Swapping independent quantifiers (Prop. \ref{SWAP}).

\item Swapping first-order universal quantifiers (Prop. \ref{UNISWAP}).

\item Removing the slash set of a universal quantifier (Prop. \ref{UNIDEP}).

\item Removing purely existential slash sets (Prop. \ref{DEPEX}).

\item Distributing universal quantifiers over conjunctions (Prop. \ref{DISTRIBUTION}).

\end{enumerate}
\end{lemma}

\begin{proof}
a) This rule does not change the dependencies between logical operators, therefore it cannot generate new non-modest patterns.

b) Just observe that the relative order and dependencies between universal quantifiers play no role in the definitions of non-modest patterns.

c) The slash sets of universal quantifiers play no role in the definitions of non-modest patterns.

d) Generalized Henkin and coordinated patterns contain no existential quantifiers with empty slash sets; so, they cannot be produced by this rule.

If an existential quantifier with empty slash set occurs in a signalling pattern, then it must be the second quantifier $\exists y$ in a sequence $\forall x\dots\exists y\dots (\exists z/Z)$. But then, if $Y$ is any set of existentially quantified variables (or more generally any set which does not contain $x$), then also $\forall x\dots(\exists y/Y)\dots (\exists z/Z)$ is a signalling pattern. So, this rule cannot produce new signalling patterns.

The case of Henkin patterns is treated similarly to the signalling case.

e) This transformation does not change the dependencies among quantifiers, therefore it cannot produce new signalling or Henkin patterns. Since it also does not involve disjunctions, it cannot produce new coordinated patterns. All it does in terms of dependencies is to make a conjunction independent of a universal quantifier; therefore it can eliminate a generalized Henkin pattern, but not introduce a new one.
\end{proof}

We move towards a second lemma which shows that regular, modest sentences can be required to be in a ``normal form'' which has some properties which will be convenient for proving the main result.

\begin{df}
By the \textbf{depth} of a node $t$ in a tree $(T,\preceq_T)$ we will mean the cardinality of the set $\{s\in T \ | \ s\preceq_T t\}$ (the set of predecessors of $t$).
\end{df}

Thus the root of a tree will have depth $1$, the immediate successors of the root depth $2$, and so on. We are departing from the more common convention that the root have depth $0$ for technical reasons that will be clear in the proof of the main result.

\begin{df}
Let $\varphi$ be an $IF$ sentence, and $\circ$ an occurrence of a connective in $\varphi$. We say that $\circ$ is \textbf{frontline} in $\varphi$ if 

1) $\circ$ has at least one quantifier in its scope (i.e., there is at least one quantifier occurrence $(Qv/V)$  such that $\circ\prec_\varphi (Qv/V)$)

2) if $\circ'$ is a connective in the scope of $\circ$ (i.e $\circ\prec_\varphi\circ'$) then no quantifier is in the scope of $\circ'$.
\end{df}

\begin{lem}[Normalization] \label{LEMMANORM}
Let $\varphi$ be a regular, modest sentence.  Then 
  we can transform $\varphi$ into an equivalent regular, \CORR{modest} sentence $\varphi'$ satisfying the following requirements:
\begin{enumerate}
\item \CORR{All universal quantifiers have empty slash sets.}
\item \CORR{Nonempty slash sets contain at least one universally quantified variable.}
\item If $\circ$ is a binary connective that is frontline in $\varphi'$, and $\circ \prec_{\varphi'} \forall c \prec_{\varphi'} (\exists d/D)$, then  $d$ depends on $c$ (i.e. $c\notin D$). 

\end{enumerate}
\end{lem}

\begin{proof}

1) and 2) can be ensured by applying the rules of Prop. \ref{UNIDEP} and \ref{DEPEX} (which preserve modesty, by lemma \ref{SIMPLERULESPRESERVEMODESTY} a)-b) ).

3)
\CORR{Let $\circ \prec_{\varphi} \forall c \prec_{\varphi} (\exists d/D)$ be operators in $\varphi$ such that $\circ$ is frontline and $c\in D$.} 
We show how to push $\forall c$ below $(\exists d/D)$ by quantifier swapping. By lemma \ref{SIMPLERULESPRESERVEMODESTY} a)-b), the sentence $\varphi'$ obtained in the end is still modest (and, of course, regular). 

 First of all, notice that there are no connectives between $\circ$ and $(\exists d/D)$ (since $\circ$ is frontline); so, we can use the quantifier swapping rule to rearrange the part of the sentence strictly comprised between $\circ$ and $(\exists d/D)$ in Hintikka normal form (\cite{ManSanSev2011}, Theorem 5.45), as a sequence of universal quantifiers followed by a sequence of existential quantifiers.

Secondly, one can push $\forall c$ below the other universal quantifiers (using proposition \ref{UNISWAP}), until it is immediately above the sequence of existential quantifiers.

If the sequence of existential quantifiers begins with quantifiers that are independent of $\forall c$, swap them above $\forall c$. 

Then, below $\forall c$ and before $\exists d$, find the first pair of existential quantifiers $(\exists u/U),(\exists v/V)$ such that: 1) $(\exists u/U)$ is immediately above $(\exists v/V)$,  2) $u$ depends on $c$, and 3) $v$ does not depend on $c$.
Now, if $v$ depended on $u$, then $\forall c,(\exists u/U),(\exists v/V)$ would form a signalling sequence, contradicting the hypothesis that $\varphi$ is modest. So, $v$ is independent of $u$; thus, we can swap $(\exists v/V)$ above $(\exists u/U)$; then, for the same reason, we can push it above all the existential quantifiers that were between $\forall c$ and $(\exists u/U)$; and finally, above $\forall c$. Iterating this process, one can push above $\forall c$ all the existential quantifiers that are independent of $c$, including $(\exists d/D)$.  
%
%
\end{proof}

Imagine that you have a modest and normalized sentence, you extract a universal quantifier above one of its frontline connective, and in the process a generalized Henkin pattern is formed. The following lemma shows that, in this particular case, we have a good amount of information concerning the structure of the generalized Henkin pattern. In some cases the same holds for coordinated patterns of the first kind.

\begin{lem}[Consequences of normalization]\label{LEMGENHEN}
Let $\varphi$ be a regular, modest $IF$ sentence in the normal form described in lemma \ref{LEMMANORM}. Let $\psi$ be a subformula of $\varphi$ of the form        $\forall a\chi_1 \circ \chi_2$, where $a$ does not occur in $\chi_2$ and $\circ$ is a frontline connective  of $\varphi$.  Let $\varphi'$ be obtained by replacing $\psi$ with $\psi':\forall a(\chi_1 \circ (\chi_2)_{|a})$.\footnote{$\varphi'$ is equivalent to $\varphi$ by the strong extraction rule (proposition \ref{STRONGEXTRACTION}) and substitution of equivalents (proposition \ref{TEOSUBEQ}). } 

a) Suppose $\varphi'$ contains a generalized Henkin pattern  $\forall a,(\exists u/U),\forall z,(\exists v/V),\circ$, where $u$ depends on $a$ (but not on $z,v$) and $v$ depends on $z$ (but not on $a,u$). Then:
\begin{enumerate}
\item $(\exists v/V)$ is in $(\chi_2)_{|a}$.
\item  $(\exists u/U)$ is in $\chi_1$.
\item There is a quantifier $\forall b\prec_\varphi \circ$ such that $b\in V$.
\end{enumerate}

 b) Suppose $\varphi'$ contains a first kind coordinated pattern  $\forall a,\forall y,\forall z,\lor,(\exists u/U),$ $(\exists v/V)$, where $u$ depends on $y$ (but not on $a,z$) and $v$ depends on $z$ (but not on $a,y$). Then this occurrence of $\lor$ is $\circ$, and 1.,2.,3. hold. 
\end{lem}

\begin{proof}
We begin with proving a).

1) Normalization excludes the possibility that $\chi_1$ contains an existential quantifier independent of $a$; so, $(\exists v/V)$ is not in $\chi_1$. It is not above $\circ$, either, because otherwise it would be in a same branch with  $(\exists u/U)$, contradicting the fact that these two quantifiers are part of a generalized Henkin pattern. The third and last possibility is that there is a connective $\circ'\prec_{\varphi'} \circ$ such that $(\exists v/V)$ occurs in the subformula immediately below $\circ'$ which does not contain $\circ$. We show that this is impossible. Since $\forall a\prec_{\varphi}(\exists u/U)$, we see that $\forall a,(\exists u/U),\forall z,(\exists v/V),\circ'$ already formed a generalized Henkin pattern in $\varphi$: this contradicts the assumption that $\varphi$ was modest. 
So we can conclude:

(*) $(\exists v/V)$ is in $(\chi_2)_{|a}$.

2) We want to show that, instead, $(\exists u/U)$ is in $\chi_1$. It cannot be above $\circ$ (otherwise it would also be above $\forall a$, while we know it depends on it). For similar reasons, it cannot be in a branch which does not contain $\circ$. Finally, if it were in $(\chi_2)_{|a}$, we already know that it must be on a different branch with respect to $(\exists v/V)$; thus there is a connective $\circ'$ in $(\chi_2)_{|a}$ such that $\circ'\prec_{\varphi'}(\exists u/U),(\exists v/V)$. But this is impossible, since $\circ \prec_{\varphi'}\circ'$ and $\circ$ is frontline in $\varphi'$. So, $(\exists u/U)$ is in $\chi_1$.


3) Given (*) and the fact $v$ is not dependent on $a$, we can conclude that $a\in V$. The fact that  the strong extraction rule was used  implies that $V\setminus\{a\}\neq\emptyset$; i.e., there is a quantifier $Qb\prec_{\varphi'}(\exists v/V)$ such that $b\in V$. By the normalization assumption 2., $Q$ is $\forall$. By the normalization assumption 3., there is at least one connective $\circ'$ occurring between $\forall b$ and $(\exists v/V)$; since $\circ$ is frontline, we conclude that $\circ'\preceq\circ$. This immediately entails that $\forall b \prec_{\varphi'}\circ$.

The statement b)  has  a completely analogous proof, once one proves that $\circ$ is the specific occurrence of $\lor$ which is part of the coordinated pattern. We prove this point. By definition of coordinated pattern we have $\forall a\prec_{\varphi'}\lor\prec_{\varphi'} (\exists u/U),(\exists v/V)$. However, since $\forall a\prec_{\varphi'}\lor$ and $\forall a$ is immediately above $\circ$, we also have $\circ \preceq_{\varphi'}\lor$. Since $\circ \preceq_{\varphi'}\lor$, $\circ$ is frontline, and there are quantifiers in the scope of $\lor$, we must conclude that $\circ$ and $\lor$ coincide.
\end{proof}

We will say that operators  $(\forall  x/X),  \land, (\forall y/Y), (\forall z/Z), (\exists u/U), (\exists w/W)$ occurring in a syntactical tree form a \textbf{$\land$-coordinated pattern} if\\
0) $(\forall  x/X),  \land, (\forall y/Y), (\forall z/Z), (\exists u/U), (\exists w/W)$ do not all occur in a same branch\\
1) $u$ depends on $y$, but not on $x, z, w$ \\
2) $w$ depends on $z$,  but not on $x, y, u$\\
3) $\land$ is in the scope of $\forall x$, and $(\exists u/U), (\exists w/W)$ are in the scope of $\land$.\\
(These are just the conditions for coordinated patterns, restated for $\land$ instead of $\lor$).

\begin{lem}[Careful extraction does not increase the number of $\land$-coordinated patterns] \label{LEMANDCOORD}
Let $\varphi$ be a regular $IF$ sentence which is normalized according to lemma \ref{LEMMANORM}. Let $\psi$ be a subformula of $\varphi$ of the form        $(Qa/A)\chi_1 \circ \chi_2$, where $a$ does not occur in $\chi_2$ and $\circ$ is a frontline connective  of $\varphi$.  Let $\varphi'$ be obtained by replacing $\psi$ with $\psi':(Qa/A)(\chi_1 \circ (\chi_2)_{|a})$.\footnote{$\varphi'$ is equivalent to $\varphi$ by the strong extraction rule (proposition \ref{STRONGEXTRACTION}) and substitution of equivalents (proposition \ref{TEOSUBEQ}). } Then $\varphi'$ has no new  $\land$-coordinated patterns  with respect to $\varphi$. 
\end{lem}

\begin{proof}
Suppose that after replacing $\psi$, the resulting sentence $\varphi'$ has some extra $\land$-coordinated pattern. 
Obviously the new pattern must contain $(Qa/A)$.

Suppose first that $(Qa/A)$ is $\forall a$. Since in the definition of $\land$-coordinated pattern it does not matter whether the quantifiers over $y$ and $z$ occur above or below $\land$, we must conclude that our $\forall a$ plays the part of $(\forall  x/X)$ in the definition of $\land$-coordinated pattern. So the pattern contains two existential quantifiers $(\exists u/U)$ and $(\exists v/V)$ such that $a\in U$  and $a\in V$. Since $\circ$ is frontline in $\varphi$, $\chi_1$ and $\chi_2$ contain no frontline connectives. 
 Then,  $(\exists u/U)$ and $(\exists v/V)$ cannot occur both in the same $\chi_i$. We can then assume w.l.o.g. that $(\exists u/U)$ is in $\chi_1$. Since $\forall a$ occurred in the scope of $\circ$ in $\varphi$, and $a\in U$, this contradicts the assumption that $\varphi$ was normalized (see point 3. of lemma \ref{LEMMANORM}).

Suppose instead that $Q$ is $\exists$. Then we are asserting that $(\exists a/A)$ forms, in $\varphi'$, a new $\land$-coordinated pattern together with some operators, none of which occurs in the scope of  $(\exists a/A)$ in $\varphi'$. However, this is impossible, because replacing $\psi$ with $\psi'$ only changes the dependency relations between $(\exists a/A)$ and some operators in its scope (in $\varphi'$).
\end{proof}

By simpler arguments, one can see that also some other equivalence rules do not increase the number of $\land$-coordinated patterns.

\begin{lem}\label{LEMANDCOORD2}
Let $\varphi$ be a regular $IF$ sentence, and $\varphi'$ be obtained by applying to $\varphi$ a sequence of the following transformations:
\begin{enumerate}
\item Swapping independent quantifiers (prop. \ref{SWAP})
\item Swapping universal quantifiers (prop. \ref{UNISWAP})
\item Distributing a universal quantifier over a conjunction (prop. \ref{DISTRIBUTION}).
\end{enumerate}
Then $\varphi'$ contains no new $\land$-coordinated patterns with respect to $\varphi$.
\end{lem}

Suppose we have a normalized sentence $\varphi$, and suppose we extract a quantifier above one of the frontline connectives, say $\circ$. It might then happen that in the new sentence $\varphi'$ there are no more quantifiers in the scope of $\circ$: this implies that the set of frontline connectives is now different ($\circ$ is not anymore frontline, and possibly there is a new frontline connective $\circ'$). As a consequence, it may happen that $\varphi'$ is not in normal form. In the next lemma, we point out that it is then possible to renormalize $\varphi'$ in a way that impacts only ``locally'' the structure of the sentence. 

\begin{df}
Let $\varphi$ be a sentence and $c$ be a node in its syntactical tree. If there is a frontline connective $\circ$ such that $\circ\prec_\varphi c$, then we say that $c$ is in the \textbf{lower part} of $\varphi$; otherwise, we say it is in the \textbf{upper part} of $\varphi$.
\end{df}

\begin{lemma}[Local renormalization]\label{LEMRENORM}
Let $\varphi$ be a regular $IF$ sentence which is normalized in the sense of lemma \ref{LEMMANORM}, and let $\circ$ be a frontline connective of $\varphi$. Suppose $\varphi'$ is obtained by extracting a quantifier above $\circ$ by using proposition \ref{STRONGEXTRACTION}. Then there is a regular, normalized $IF$ sentence $\varphi''$ such that $\varphi'\equiv\varphi''$, and which has the same upper part as $\varphi'$. 

Furthermore, a) if $\varphi'$ is modest, then also $\varphi''$ is, and b) $\varphi'$ and $\varphi''$ have the same number of $\land$-coordinated patterns.
\end{lemma}

\begin{proof}
As a first case, suppose that in $\varphi'$ there are no frontline connectives $\circ'\preceq_{\varphi'} \circ$. This means that the frontline connectives of $\varphi'$ are a proper subset of those of $\varphi$; thus, $\varphi'$ already satisfies condition 3. of normalization. Then one can proceed as in the normalization lemma to ensure that conditions 1. and 2. hold within the scope of $\circ$. It is obvious that the resulting sentence $\varphi''$ satisfies the statement.

Suppose instead that there is a connective $\circ'\preceq_{\varphi'}\circ$ which is frontline in $\varphi'$. We observe that, by definition of frontline connective, such $\circ'$ is unique (it is either $\circ$ itself, or the connective of maximum depth among those that have $\circ$ in their scopes). Observe then that each node in the scope of $\circ'$ is in the lower part of $\varphi'$. Now apply the transformations described in the normalization lemma (\ref{LEMMANORM}) to the subformula which has $\circ'$ as its most external operator; clearly, these transformations do not affect the nodes outside the subformula; so, in particular, they do not affect the upper part of $\varphi'$. We can then take $\varphi''$ to be the formula thus obtained.

For a), we have already observed in lemma \ref{LEMMANORM} that the transformations used in the normalization process preserve modesty.

For b), observe once more that the transformations used in the normalization process (i.e., the quantifierswapping  rules \ref{SWAP} and \ref{UNISWAP}) do not change the dependencies between quantifiers and connectives; therefore, they cannot create or eliminate $\land$-coordinated patterns. 
\end{proof}

\begin{df}
By the \textbf{connective-depth} of a node $t$ in a tree $(T,\preceq_T)$ we will mean the cardinality of the set $\{\circ\in T \ | \ \circ = \land \text{ or } \lor \text{, and } \circ\preceq_T t\}$.
\end{df}

\noindent Notice: if $\circ=\land,\lor$ occurs in $T$, then $\circ$ has at least connective-depth $1$.

\begin{teo} [Theorem \ref{MODESTTREE} in the main text]
a) All regular, modest $IF$ tree prefixes are in FO.\\
b) Every regular, modest $IF$ sentence can be transformed, by means of equivalence rules, into a first-order sentence. 
\end{teo}

\begin{proof}
Let $T$ be a regular, modest tree and $\varphi$ be a sentence which is a completion of $T$. We can also assume without loss of generality that $\varphi$ is normalized as in Lemma \ref{LEMMANORM}. We want to show that $\varphi$ is equivalent to some other sentence which is prenex form, and still regular and modest (thus, primary). By Sevenster's result (Prop. \ref{SEVCLASSIFY}), this last sentence is then equivalent to some first-order sentence. This yields both a) and b).

We face a problem: extraction of quantifiers does not preserve \emph{in general} the modesty of the tree (or the sentence). We will thus need to choose carefully the order in which the extractions are performed. And in one case (case 2c.C below) we will not be able to apply extraction at all; in that case, we will have to push some quantifiers in the opposite direction, using quantifier distribution and swapping. However, each time this case is reached, we will reduce by one the number of $\land$-coordinated patterns occurring in the sentence; and lemma \ref{LEMANDCOORD} will guarantee that this number is never increased throughout the procedure. So, the transformations given in case 2c.C will be applied at most a finite number of times; this guarantees that the procedure we describe terminates, yielding a regular, modest sentence in prenex form.

To be more precise, our procedure will consist in applying $2n$ transformations
\[
\varphi = \varphi_0   \stackrel{E_1}{\mapsto}   \varphi_1  \stackrel{N_1}{\mapsto}  \varphi_2 \stackrel{E_2}{\mapsto}   \dots  \stackrel{E_n}{\mapsto}\varphi_{2n-1} \stackrel{N_n}{\mapsto} \varphi_{2n}
\]
each of which preserves truth-equivalence and modesty; each consecutive application of $E_i$ followed by $N_i$ will also preserve normalization. Typically, $E_i$ will be an application of the strong extraction rule (prop. \ref{STRONGEXTRACTION}, possibly preceded by a renaming, prop \ref{VARIANT}), applied to a subformula which has a frontline connective as its most external operator; while $N_i$ will be a local renormalization (lemma \ref{LEMRENORM}). Exceptionally, when case 2c.C, described below, is reached, $E_i$ will be a sequence of applications of quantifier swapping (prop. \ref{SWAP} and \ref{UNISWAP}) and quantifier distribution (prop. \ref{DISTRIBUTION}). The  sentence $\varphi_{2n}$ which is reached at the end will be the sentence that we need (i.e., it will be regular, modest and prenex). To understand that we can really reach such a sentence in a finite number of steps, consider the following counters: 
\[
k: \text{sum of the connective-depths of frontline connectives in the current sentence}
\]
\[
l: \text{number of $\land$-coordinated patterns in the current sentence.}
\]
As long as $k>0$, it is possible to extract quantifiers that lie immediately below some frontline connective; therefore the procedure can go forward. If $k$ is $0$, instead, there are no more frontline connectives (remember that occurrences of connectives have at least connective-depth $1$!); thus, the current sentence is prenex (and regular and modest, since the $E_i$ and $R_i$ preserve these two properties), and so we are done. Each quantifier extraction reduces the number of quantifiers occurring below a frontline connective without increasing $k$; when the last quantifier occurring under such a connective is extracted, $k$ is reduced at least by one. When quantifier distribution (case 2c.C) is applied, then the number of quantifiers under some connective increases, and possibly also $k$; however, $l$ is decreased by $1$. It will be seen that $l$ is never increased; and that, once $l$ is $0$, the case 2c.C cannot be reached anymore. So, quantifier distribution is applied only a finite number of times; after this, the counter $k$ will not increase anymore. 

Let us fix some notation to describe what happens when $E_i$ is applied. So, let $\circ$ be a frontline connective of $\varphi_{2i-1}$. 
Since $\circ$ is frontline, immediately below $\circ$ there occurs a quantifier $(Qa/A)$. 
Let $L$ be the subformula immediately below $(Qa/A)$, and $R$ be the subformula immediately below $\circ$ that does not contain this occurrence of $(Qa/A)$. First, if needed, we use renaming (proposition \ref{VARIANT}) to replace $a$ with a new variable (but in the following we will keep writing $a$); then we use the strong extraction rule, which also replaces $R$ with $R' = R_{|a}$:

\Tree [.\vdots\\$\circ$ [.$Qa/A$ $L$ ] $R$  ]
\Tree [.\vdots\\$Qa/A$ [.$\circ$ $L$ $R'$ ]  ]

We call $\varphi_i$ the resulting sentence. 
By lemma \ref{LEMMANORM}, the number of $\land$-coordinated patterns does not increase with this transformation. In the last part of the proof we will check that the  sentence resulting after the extraction is still modest (except in case 2c.C, which is treated differently). Before that, we underline that the sentence $\varphi_{2i}$ resulting after extraction may fail to be normalized; normalization is then restored when applying $N_i$, that is, the  local renormalization described in lemma \ref{LEMRENORM}. This lemma tells us that the renormalized sentence has  the same number of $\land$-coordinated patterns and the same upper part as $\varphi_{2i}$, that is, in particular, the same frontline connectives. So both counters $k$ and $l$ are not increased by $N_i$. 

In the rest of the proof we check that, after extraction, the resulting sentence is still modest, and, in the one case when extraction does not preserve modesty (case 2c.C), we explain what transformations should be applied. 

1) Suppose $Q=\exists$.

Then the new sentence $\varphi_{2i}$ obtained after quantifier extraction is not \underline{signalling} (the existential quantifiers in $R'$ are either first-order, in which case they cannot play the role of the rightmost quantifier of a signalling pattern, or they have nonempty slash set, in which case $a$ has been added to their slash set in $R'$, and they cannot receive signals from $\exists a$). 

Suppose $\varphi_{2i}$ is \underline{Henkin}. Then there are quantifiers $\forall x,\forall y,$ $(\exists b/B)$ such that $\forall x,(\exists a/A),\forall y,(\exists b/B)$ form a Henkin pattern in $\varphi_{2i}$ (say, $a$ depends on $x$). But then  the logical operators $\forall x,(\exists a/A),\forall y,(\exists b/B'),\circ$ (where \CORR{$B' = B\setminus \{a\}$}) formed a generalized Henkin pattern in $\varphi_{2i-1}$: contradiction.
 
$\varphi_{2i}$ cannot be \underline{generalized Henkin} nor \underline{co\phantom{g}\hspace{-5pt}ordinated}, because in $\varphi_{2i}$  there are no new universal-existential dependence pairs, no new universal-connective dependence pairs and no new connective-existential dependence pairs with respect to $\varphi_{2i-1}$.

2) Suppose $Q=\forall$.

 2a) If $\varphi_{2i}$ is \underline{signalling}, then this must be witnessed by $\forall a$ itself (otherwise, the signalling pattern would have been already in the original tree) and by two existential quantifiers $(\exists u/U),(\exists v/V)$ occurring in $R'$ and such that $a\notin U$, $a\in V$, $u\notin V$. 
We can conclude that $U$ is empty: otherwise, after applying strong extraction, we would have $a\in U$.
 Furthermore, we have that the slash set of $v$ contains another variable $b$ (otherwise, the strong extraction rule would have preserved the empty slash set); by our normalization assumptions, we can assume $b$ is universally quantified above $\circ$. But then, the quantifiers over $b,u,v$ prove that  $\varphi_{2i-1}$ was signalling: contradiction.


2b) Suppose $\varphi_{2i}$ is \underline{Henkin}. Then there are quantifiers $(\exists u/U),\forall b,(\exists v/ V)$ in $\varphi_{2i}$ such that $u$ depends on $a$ but not on $b$ nor $v$; and $v$ depends on $b$ but not on $a$ nor $u$.

Notice that, since $(\exists u/U)$ and $(\exists v/ V)$ are in the same branch in $\varphi_{2i}$, they are also in the same branch in $\varphi_{2i-1}$. But they cannot be above $\circ$ (otherwise $\forall a,(\exists u/U),\forall b,(\exists v/ V)$ would form a Henkin pattern already in $\varphi_{2i-1}$); and for similar reasons they cannot be in $L$. So, $(\exists u/U)$ and $(\exists v/ V)$ are both in $R'$.




Since strong extraction is used, $a\notin U$ implies that $U = \emptyset$;
but then, since $u$ does not depend on $b$ nor on $v$, we must conclude that $\exists u\prec_{\varphi_{2i}} (\exists v/V), \forall b$, and that $u\in V$. From the former it also follows that $\forall a \prec_{\varphi_{2i}}(\exists v/V)$; so, since $v$ does not depend on $a$, we must conclude that $a\in V$.

The facts that $a\in V$, that $(\exists v/V)$ is in $R'$, that we have applied strong extraction, together with the normalization assumption 2. imply that $V$ contains at least one variable $v$ which is universally quantified and distinct from $a$;   that is, in $\varphi_{2i}$ there is a universal quantifier $\forall c$ (distinct from $\forall a$) occurring above $(\exists v/ V)$ and such that $c \in V$.

Now notice that $\forall c$ cannot occur above $\exists u$: if it did, then $\forall c,\exists u, (\exists v/V)$ would be a signalling pattern already occurring in $\varphi_{2i-1}$: contradiction.  So, $\exists u \prec_{\varphi_{2i-1}} \forall c\prec_{\varphi_{2i-1}}(\exists v/V)$: therefore, $\forall c$ is in $R'$, hence in the scope of $\circ$. But then, since we assumed that $\varphi_{2i-1}$ was normalized, we must have $c\notin V$: contradiction.

2c) Suppose that after applying strong extraction to a frontline connective of $\varphi_{2i}$ we obtain a sentence $\psi$ which is \underline{generalized Henkin}, as witnessed by quantifiers $\forall a,(\exists u/U),\forall z,(\exists v/V)$ such that $u$ depends on $a$ (but not on $z,v$) and $v$ depends on $z$ (but not on $a,u$). Since we are assuming that $\varphi_{2i}$ is regular, modest and normalized, lemma \ref{LEMGENHEN}, a) tells us that $(\exists u/U)$ is in $L$, $(\exists v/V)$ is in $R'$, and there is a quantifier $\forall b \prec_{\psi}\circ$ such that $b\in V$. These observations can be summarized by saying that $\psi$ has the following form:

\Tree [.$\vdots$\\$\forall b$\\$\vdots$\\$\forall a$ [.$\circ$ [.$\vdots$\\$(\exists u/U)$\\$\vdots$ ] [.$\vdots$\\$(\exists v/\{a,b,\dots\})$\\$\vdots$ ] ]  ]

 The quantifier $\forall z$ is not shown in the picture; we only know that it is somewhere above $(\exists v/V)$ (since $v$ depends on $z$). Now there are three cases. 

A) $u$ depends on $b$. In this case, $\varphi_{2i-1}$ was already generalized Henkin (as witnessed by  $\forall b,(\exists u/U),\forall z, \circ,(\exists v/V)$): contradiction. In this case, one can take $\varphi_{2i}$ to be $\psi$.

B) $u$ does not depend on $b$, and there is a disjunction $\forall b \prec_{\psi} \lor \preceq_{\psi} \circ$.
Then $\varphi_{2i-1}$ was coordinated (as witnessed by $\forall b,\forall z,\forall a,\lor,(\exists u/U),(\exists v/V)$): a contradiction. Also in this case, one can take $\varphi_{2i}$ to be $\psi$.

C) $u$ does not depend on $b$, and there is no disjunction $\forall b \prec_{\psi} \lor \preceq_{\psi} \circ$.  In this case we find no contradiction; so, instead of using quantifier extraction, we will apply some different transformations that will preserve modesty and reduce the number of $\land$-coordinated patterns by one. Observe that $\forall a,\circ,\forall b, \forall z, (\exists u/U)$ and $(\exists v/V)$ form a $\land$-coordinated pattern in $\varphi_{2i}$; therefore, $\forall a,\circ,\forall b, \forall z, (\exists u/U)$ and $(\exists v/V')$ (where $V'= v\setminus\{a\}$) already formed a $\land$-coordinated pattern in $\varphi_{2i-1}$. 
Now notice that there are no existential \CORR{quantifiers} depending on $b$ and above $(\exists u/U)$ and $(\exists v/V)$ -- otherwise the tree would be either signalling or Henkin.
So, $\forall b$ can be pushed down by quantifier swapping (prop. \ref{SWAP} and \ref{UNISWAP}) and distribution (\ref{DISTRIBUTION}), until it goes below $\circ$.  Call $\varphi_{2i}$ the resulting sentence. It should be clear that $\varphi_{2i}$ is still modest (the only depedency relations that have changed are those between $\forall a$ and some occurrences of $\land$; these changes cannot create any new non-modest patterns). Also, by lemma \ref{LEMANDCOORD2}, the number $l$ of $\land$-coordinated patterns did not increase in the process, while the operators  $\forall a,\circ,\forall b, \forall z, (\exists u/U)$ and $(\exists v/V)$ do not form such a pattern anymore; so $l$ has decreased by one.\footnote{At later stages of the procedure, the steps described in this case 2C.c will be applied to each of the universal quantifiers that play the same role as $\forall b$ (with respect to $\forall a$ and $\circ$). Only after pushing down all these quantifiers, it will be possible to extract $\forall a$ above $\circ$.}  

2d) Now, suppose instead $\varphi_{2i}$ is \underline{coordinated, first kind}. It means that, in the new tree, $\forall a$ plays either the role of $\forall x$ or that of $\forall y$ in the definition of coordinated tree. 

In the first case, we observe that  in $\varphi_{2i}$ there are quantifiers $\forall z,\forall y$ and quantifiers $(\exists u/U),(\exists v/V)$ such that $u$ depends on $y$ but not on $a$ nor $z$, while $v$ depends on $z$ but not on $a$ nor $y$; and a disjunction $\forall a\prec_{\varphi_{2i}}\lor\prec_{\varphi_{2i}} (\exists u/U),(\exists v/V)$. Lemma \ref{LEMGENHEN}, b) then tells us that $\circ$ and $\lor$ coincide,  $(\exists u/U)$ is in $L$, $(\exists v/V)$ is in $R'$, and there is a quantifier $\forall b \prec_{\psi}\circ$ such that $b\in V$.
We then check again the three subcases A), B), C) that were already considered in case 2c); case A) and case B) work similarly as before. 
 Case C) cannot be, because we already know that $\circ$ is an occurrence of $\lor$. 

In the second case ($\forall a$ playing the role of \CORR{$\forall y$} in the definition of coordinated tree), to fix notation, $\forall a$ forms a first kind coordinated pattern in $\varphi_{2i}$ together with quantifiers $\forall x,\forall z, (\exists u/U), (\exists v/V)$ (s.t. $u$ depends on $a$ but not on $x$ nor $z$, and $v$ depends on $z$ but not on $x$ nor $a$)  and an occurrence of $\lor$. We want to show that this occurrence of $\lor$ is $\circ$. 
We first observe that $u$ depends on $a$, but $U$ is nonempty; so $\forall a \prec_{\varphi_{2i-1}}(\exists u/U)$ (otherwise quantifier extraction would have added $a$ to $U$); so  $(\exists u/U)$ is in $L$ (the left subformula under $\circ$).
Next we prove the following

\underline{Claim}:  $(\exists v/V)$ is in $R'$.

\begin{proof}
If $(\exists v/V)$ occurred above $\forall a$, then the quantifiers $\forall x,\forall z, (\exists u/U), (\exists v/V)$ would all be on a same branch, contradicting the definition of coordinated pattern.If $(\exists v/V)$ occurs in the scope of $\forall a$, then it did so already in $\varphi_{2i-1}$, which is impossible by point 3. of normalization. Finally, it might be that there is a connective $\circ'\prec_{\varphi_{2i}}\circ$ such that $(\exists v/V)$ occurs in the subformula immediately below $\circ'$ which does not contain $\circ$. But we know that $(\exists u/U)$ is in $L$, so that already in $\varphi_{2i-1}$ we had $\forall a \prec_{\varphi_{2i-1}} (\exists u/U)$. This entails that $\forall x,\forall a,\forall z, (\exists u/U), (\exists v/V),\lor$ was already a coordinated pattern in $\varphi_{2i-1}$, contradicting the assumption that $\varphi_{2i-1}$ is modest.
\end{proof}

But then $\circ$ is the unique connective which has   $(\exists u/U)$ in its left subformula and $(\exists v/V)$ in its right subformula; so $\circ$ is the occurrence of $\lor$ which is part of the coordinated pattern. This and the fact that   $\forall a \prec_{\varphi_{2i-1}} (\exists u/U)$   imply that $\forall x,\forall z, \forall a,(\exists u/U), (\exists v/V),\circ$ is a coordinated pattern already in $\varphi_{2i-1}$: contradiction.


2e) Suppose $\varphi_{2i}$ is \underline{coordinated, second kind}. Then, $\varphi_{2i-1}$ contains a coordinated pattern $\forall a,\forall y,\forall z,\lor,(\exists u/U),(\exists v/V)$ of second kind; so,  $(\exists u/U),(\exists v/V)$ both occur in a same subformula under this occurrence of $\lor$. By definition of coordinated pattern, $(\exists u/U),(\exists v/V)$ must occur in different branches; this means that there is exactly one connective $\circ'$ which has  $(\exists u/U)$ in its left subformula and $(\exists v/V)$ in its right subformula (this is also the connective of maximum depth among those that have $ (\exists u/U)$ and $(\exists v/V)$ in their scope).

Now suppose first that both $u$ and $v$ do not depend on $a$ (i.e., $\forall a$ plays the role of $\forall x$ in the definition of coordinated pattern). Then, again by definition of coordinated pattern, we also have that $\forall a\prec_{\varphi_{2i}}\lor$; and furthermore, $\forall a \prec_{\varphi_{2i}} (\exists u/U),(\exists v/V)$, and consequently $\circ \prec_{\varphi_{2i}} (\exists u/U),(\exists v/V)$. Then it follows that $\circ \prec_{\varphi_{2i}}\circ'$. But then, since $\circ$ is frontline and there are quantifiers in the scope of $\circ'$, we must conclude that $\circ'$ is $\circ$. But then $\lor \prec_{\varphi_{2i}}\circ$; and since $\forall a$ is the immediate predecessor  of $\circ$ in $\varphi_{2i}$, we also have $\lor\prec_{\varphi_{2i}}\forall a$, contradicting our initial assumption.

Suppose instead that  $\forall a$ plays the role of $\forall y$ in the definition of coordinated pattern, that is: $u$ depends on $a$, and $v$ depends $z$, $\forall y\prec_{\varphi_{2i}}\lor$ and neither $u$ nor $v$ depend on $y$. Since $u$ depends on $a$, and $U$ is nonempty ($y\in U$) we must conclude that $\forall a \prec_{\varphi_{2i-1}}(\exists u/U)$ (otherwise the strong extraction would have added $a$ to the slash set of $u$). Observe then that $\forall a,\forall y,\forall z,\lor,(\exists u/U),(\exists v/V)$ was already a second kind coordinated pattern in $\varphi_{2i-1}$: contradiction.
\end{proof}

\section{Proof of the Extension Lemma (lemma \ref{EXTENDLEMMA})} \label{APPEXTLEM}

We point out the simple fact that, if $U$ extends $T$, then $U$ can be ``constructed'' adding one by one new logical operators to $T$, and updating the slash sets.

\begin{lem}\label{LEMCONSTRUCT}
Suppose $U$ extends $T$. Then there is a finite sequence $T=T_0, T_1,\dots,T_{n+1} =U$ of regular tree prefixes  such that, for each $i=0..n$, $T_{i+1}$ is obtained from $T_i$ in one of two ways:
\begin{enumerate}
\item $T_{i+1}$ is obtained by adding to $T_i$ a connective $\circ$ and one extra gap below it.
\item $T_{i+1}$ is obtained by adding to $T_i$ a quantifier $(Qv/V)$, and adding $v$  to some of the slash sets in the scope of $(Qv/V)$.
\end{enumerate}
\end{lem}

\noindent Notice that we must add the quantifiers starting from those of maximum depth, if we want to keep track correctly of what variables must be added to the slash sets. And on the other hand, if $c,d,e$ are in $U\setminus \mu[T]$ and $c,e$ are to occur in distinct branches under the connective $d$, then it is possible to add both operators $c$ and $e$ only after $d$ has been added (otherwise we do not have enough branches). These two factors make the construction not completely obvious. 

\begin{proof}
We prove, by induction on the depth of the operators in $U\setminus \mu[T]$, the claim together with the additional statement that $U$ is an extension of $T_{i+1}$ via a function $\mu_{i+1}$. Suppose $T_i$ has been constructed. Pick a connective $c$ of \emph{minimum} depth in $U\setminus \mu[T_i]$, if there is one. Add it to $T_i$ in the same position in which it occurs in $U$; add a gap below $c$ so that $c$ has exactly two nodes below (so that the resulting $T_{i+1}$ is a syntactical tree). Let $\mu_{i+1} = \mu_i\cup\{(c',c'')\}$ (where $c'$ is the occurrence of $c$ in $T_{i+1}$, and $c''$ the occurrence in $U$).
It is straightworward to check that, if $T_i$ is regular, also $T_{i+1}$ is; and that, if $U$ is an extension of $T_i$, then it is also an extension of $T_{i+1}$.

When a tree $T_j$ is reached such that $U\setminus\mu_j[T_j]$ contains no connectives, we begin treating the quantifiers. Pick a quantifier $(Qv/V)$ of \emph{maximum} depth in  $U\setminus\mu_j[T_j]$ and add it to $T_j$ in the same position in which it occurs in $U$. Then, for each quantifier $(Q'u/U)$ in the scope of $(Qv/V)$, add $v$ to the slash set of  $(Q'u/U)$ if and only if $v$ is in the slash set of $\mu_j( (Q'u/U) )$. 
 Let $\mu_{j+1} = \mu_j\cup\{((Qv/V),(Qv/V))\}$ (where again the former is the occurrence in $T_{j+1}$, and the latter the occurrence in $U$).
Regularity of $T_{j+1}$, and the fact that $U$ is an extension of $T_{j+1}$, can be checked as above.
\end{proof}

We will need some lemmas from the literature, which describe some regularities in the interaction between teams and $IF$ formulas. Given two teams $X$, $Y$ of disjoint domains, we define $X\times Y:=\{s\cup t \ | \ s\in X, t\in Y\}$. Notice that this is not the cartesian product of $X$ and $Y$.

\begin{lem}[\cite{ManSanSev2011}, Theorem 5.5] \label{LEMCART} 
Let $\psi$ be an $IF$ formula, $M$ a structure, $X,Y$ teams with $dom(X)\supseteq FV(\psi)$, $dom(X)\cap dom(Y)=\emptyset$. Then $M,X\models \psi$ iff $M,X\times Y\models \psi$. 
\end{lem}

\begin{lem}[\cite{ManSanSev2011}, Theorem 5.8] \label{LEMEXTTEAM}
 Let $\psi$ be an $IF$ formula, $M$ a structure, $X,Y$ teams with $dom(X)\supseteq FV(\psi)$ and $dom(Y) = dom(X)\cup V$, where $V$ is a set of variables that do not occur in $\psi$ nor $dom(X)$. Then $M,X\models \psi$ iff $M,Y\models \psi_{/V}$.
\end{lem}

\begin{lem}[\cite{ManSanSev2011}, Theorem 5.22b] \label{LEMEASIER} 
For all $IF$ formulas $\psi$ and $M$ and $X$, if $M,X\models (\exists v/V)$ and $W\subseteq V$, then $M,X\models (\exists v/W)$. 
\end{lem}

The following lemma is crucial for keeping under control the signalling phenomena. The key idea is that adding a subformula of the form $v=c$ prevents the variable $v$ being used as a signal.

\begin{lem}\label{LEMADDQF}
Let $\psi$ be an $IF$ formula, $v$ a variable which does not occur in $\psi$, $c$ a constant symbol. Let $\psi'$ be a formula obtained from $\psi$ by adding $v$ to some of the slash sets (as a limit case, $\psi'=\psi$).  Then we have the following equivalences:
\begin{enumerate}
 \item $\psi\equiv_v (\exists v/V)(v=c \land \psi')$.
\item $\psi\equiv_v (\forall v/V)\psi'$.
\end{enumerate}
\end{lem}

\begin{proof}
1) Let $M$ be a structure which interprets the signature of $\psi$ and the constant $c$. Let $X$ be a team s.t. $dom(X)\supseteq FV(\psi)$ and $v\notin dom(X)$. 

Suppose $M,X\models \psi$. Let $F:dom(X)\rightarrow dom(M)$ be the constant function such that $F(s)=c^M$ for each $s\in X$. Then obviously $M,X[F/v]\models v=c$. From  $M,X\models \psi$ and the fact that $v\notin dom(X)$, we get by Lemma \ref{LEMEXTTEAM} that $M,X[F/v]\models \psi_{/v}$. Then, by lemma \ref{LEMEASIER}, we get  $M,X[F/v]\models \psi'$. By the semantical clauses, we can conclude that $M,X\models (\exists v/V)(v=c \land \psi')$.

Suppose $M,X\models (\exists v/V)(v=c \land \psi')$. Then there is $F:dom(X)\rightarrow dom(M)$ such that $M,X[F/v]\models v=c \land \psi'$. From the semantical clauses we get $M,X[F/v]\models \psi'$; from this, by lemma \ref{LEMEASIER} we get $M,X[F/v]\models\psi$. Since $M,X[F/v]\models v=c$, we know that $F$ is the constant function that picks $c^M$. But then $X[F/v] = X\times \{(v,c^M)\}$. By this fact,  $M,X[F/v]\models\psi$ and lemma \ref{LEMCART} we get $M,X\models \psi$.

2) Let $M$ be a structure which interprets the signature of $\psi$. Let $X$ be a team s.t. $dom(X)\supseteq FV(\psi)$ and $v\notin dom(X)$. 

Assume $M,X\models \psi$. Since $v$ is not $dom(X)$ nor in $\psi$ we can use lemma \ref{LEMEXTTEAM} to obtain $M,X[M/v]\models\psi_{/v}$. By lemma \ref{LEMEASIER}, then, we obtain  $M,X[M/v]\models\psi'$. So $M,X\models \forall v\psi'$.

Suppose $M,X\models \forall x\psi'$. Then  $M,X[M/v]\models\psi'$. By lemma \ref{LEMEASIER} we get $M,X[M/v]\models\psi$. Since $X[M/v] = X \times (\{\emptyset\}[M/v])$, and $v$ does not occur in $\psi$, by lemma \ref{LEMCART} we get $M,X\models \psi$.
\end{proof}

For our purposes, part 1. of the previous lemma must still be refined. Given an $IF$ formula $\psi$, a variable $v$ that does not occur in $\psi$, and a constant $c$, we denote as $\psi^c$ the formula obtained from $\psi$ by replacing each maximal quantifier-free subformula $\alpha$ of $\psi$ with $v=c \land \alpha$.

\begin{lem}\label{LEMCFORMULA}
Let $\psi$ be an $IF$ formula, $v$ a variable which does not occur in $\psi$, $c$ a constant symbol. Let $\psi'$ be a formula obtained from $\psi$ by adding $v$ to some of the slash sets (as a limit case, $\psi'=\psi$). Then $\psi \equiv_v (\exists v/V)(\psi')^c$.
\end{lem}  

\begin{proof}
By Lemma \ref{LEMADDQF}, 2., it suffices to prove that, for all $IF$ formulas $\theta$ in which $v$ 
does not occur, $v=c\land \theta \equiv \theta^c$.
We do this by induction on $\theta$. If $\theta$ is a quantifier-free formula, then $\theta^c$ is $v=c\land \theta$, and we are done. If $\theta$ is $\eta\land\chi$, then it is easy to see that $v=c\land (\eta\land\chi)$ is equivalent with $(v=c\land \eta)\land(v=c\land\chi)$. By induction hypothesis, this is equivalent to $\eta^c \land \chi^c$, which is $\theta^c$. When $\theta$ is $\eta\lor\chi$ we proceed analogously using the equivalence of  $v=c\land (\eta\lor\chi)$ and $(v=c\land \eta)\lor(v=c\land\chi)$. In case $\theta=(Qu/U)\eta$, notice that $v=c\land(Qu/U)\eta\equiv (Qu/U)(v=c\land\eta)$ by quantifier extraction (Proposition \ref{STRONGEXTRACTION}), and apply the inductive hypothesis.
\end{proof}

\begin{lem}[Extension Lemma for tree prefixes, lemma \ref{EXTENDLEMMA} in the main text]
Suppose $U$ extends $T$, and $\varphi$ is a completion of $T$. Then there are a completion $\varphi'$ of $U$ and, for every structure $M$ suitable for $\varphi$, an expansion $M'$ of $M$ such that $M\models\varphi$ iff $M'\models\varphi'$.  
\end{lem}

\begin{proof}
Let $M'$ be identical to $M$, except that it interprets a constant symbol $c$ which was not in the signature of $M$. Obviously $M\models\varphi$ iff $M'\models\varphi$.
 If $U$ extends $T$, then by lemma \ref{LEMCONSTRUCT} there is a sequence $T=T_0, T_1,\dots,T_{n+1} =U$ of regular tree prefixes such that each $T_{i+1}$ is obtained from $T_i$ by adding one connective (plus a gap) or a quantifier (plus adding the newly quantified variable to some of the slash sets in the scope of the new quantifier). Let $e_0:=e$. We prove, by induction on $i$, that, given any  sentential completing function $e_i$ for $T_i$, there is a completing function $e_{i+1}$ for $T_{i+1}$ such that $M'\models\hat e_i(T_i)$ iff $M'\models\hat e_{i+1}(T_{i+1})$. The formula $\hat e_{n+1}(T_{n+1})$ is the formula $\varphi'$ required by the statement of the theorem. 

Suppose $T_i$ and $T_{i+1}$ differ only in that one subtree $B$ of $T_i$ is replaced by $(Qv/V)B'$ in $T_{i+1}$ ($B$ may be empty); and that $B'$ differs from $B$ only for the addition of variable $v$ to some slash sets. 
  For each branch $P$ of $T_i$, call $P'$ the corresponding branch of $T_{i+1}$.\footnote{More precisely: if $P'$ is a branch of $T_{i+1}$ that does not intersect the above mentioned occurrence of $(Qv/V)$, then it is associated to an identical branch $P$ of $T_i$. If instead $P'= S_1(Qv/V)S_2$ contains the occurrence of $(Qv/V)$, it corresponds to a branch $P = S_1S_2$ of $T_i$.}

Let $e_i$ be a completing function for $T_i$. Define $e_{i+1}$ as the completing function which assigns, to each branch $P'$ of $T_{i+1}$, the formula $(v=c \land e_i(P))$ in case $P'$ intersects $B'$ and $Q=\exists$; and, if $Q=\forall$, just formula $e_i(P)$.
Then (using lemma \ref{LEMCFORMULA} in case $Q=\exists$, and lemma \ref{LEMADDQF}, 2. if $Q=\forall$) we have:
\[
M',X\models \hat e_i(B) \ 
\Leftrightarrow \ 
M',X\models (Qv/V)\hat e_{i+1}(B').
\]
for any suitable team $X$. 


Thus, by substitution of equivalents (\ref{TEOSUBEQ} plus \ref{TEOEQSENT}), $e_{i+1}$ is the completing function for $T_{i+1}$ that we were looking for.

Suppose instead that $T_{i+1}$ differs from $T_i$ in that a certain subtree $B$ is turned into a tree $([\phantom a] \land B)$; the ordering of the conjuncts is unimportant. For any completing function $e_i$ of $T_i$, we can define the completing function $e_{i+1}$ for $T_{i+1}$ as that function which differs from $e$ only in that it assigns $\forall x(x=x)$ to the branch we marked
 with a $[\phantom a]$. Then it is clear that  
\[
M\models \hat e_i(T_i)
\Leftrightarrow 
M\models \hat e_{i+1}(T_{i+1}).
\]
Notice that this works also in the special case in which $B$ consists just of a gap symbol.

The case that $T_{i+1}$ differs from $T_i$ in that a certain subtree $B$ is turned into a tree $([\phantom a] \lor B)$ can be treated analogously, using $\forall x(x\neq x)$ instead of $\forall x(x=x)$.

\end{proof}

\section{Syntactical manipulation of trees} \label{APPMANTREE}

In a previous draft (\cite{Bar2016}) of the present paper we proved our syntactical results by means of rules that manipulate tree prefixes rather than formulas. We briefly account here for this different approach.

In \cite{Sev2014}, the analysis of low complexity quantifier prefixes is based on a somewhat intuitive notion of equivalence of prefixes. Two quantifier prefixes $R$, $S$ are defined to be equivalent if, whenever the same quantifier-free formula $\psi$ is postfixed to them, one obtains truth-equivalent formulas $R\psi\equiv S\psi$. This notion of equivalence has, for quantifier prefixes, two good properties: 1) it preserves complexity, in the sense of the above definitions; and 2) prefixes can be manipulated, up to equivalence, by means of manipulation rules that are formally identical to equivalence rules for $IF$ formulas. However, it seems to us that there is no reasonable notion of equivalence of tree prefixes which satisfies both 1) and 2) (although, we lack a formal proof of this statement). 

To give an example of what troubles can arise when manipulating tree prefixes (even in the first-order case), consider applying quantifier extraction: 

\vspace{5pt}

\Tree  [.$\forall x$  [.$\lor$ [.$\exists y$ $[$\phantom{a}$]$ ] [.$[$\phantom{a}$]$ ] ]  ]
\Tree  [.$\forall x$  [.$\exists y$ [.$\lor$ [.$[$\phantom{a}$]$ ] [.$[$\phantom{a}$]$ ] ]  ] ]

\vspace{5pt}

\noindent This should be a legitimate equivalence rule for tree prefixes, according to the requirement 2); notice indeed that our notion of complexity only allows attaching, to the right gap of the first tree, formulas containing at most the free variable $x$; and any $IF$ sentence obtained completing (with sentential completing functions) such a tree can undergo this kind of quantifier extraction. But this tranformation does not preserve complexity, that is, it does not satisfy requirement 1); indeed, the second tree is more complex than the first, because in its right gap it is possible to attach formulas with free variables $x,y$ (not only $x$).

It would be possible, however, to define, instead of an equivalence relation, an \emph{ordering} relation, or \emph{reduction}, between tree prefixes, in a way that the usual syntactical transformations, when applied to tree prefixes, are reductions; and so that the following property is satisfied: 1') if $T$ reduces to $T'$, then C($T$) $\subseteq$ C($T'$). In \cite{Bar2016} we fully developed this approach, up to a prenex form theorem. Here, for reasons of space, we chose the more direct approach of working directly with completions of trees; that is, we fill the gaps of the tree under consideration with letters denoting generic quantifier free formulas, and we apply the usual equivalence rules to the resulting sentence:

\vspace{5pt}

\Tree  [.$\forall x$  [.$\lor$ [.$\exists y$ $\psi(x,y)$ ] [.$\chi(x)$ ] ]  ]
\Tree  [.$\forall x$  [.$\exists y$ [.$\lor$ [.$\psi(x,y)$ ] [.$\chi(x)$ ] ]  ] ]

\vspace{5pt}

\noindent Such an equivalence of sentences tells us that the tree prefix on the right is at least as expressive as the tree prefix on the left.

\end{document}